\documentclass[11pt]{article}
\usepackage[utf8]{inputenc}
\usepackage{amsmath}
\usepackage{amsfonts}
\usepackage{euscript}
\usepackage{amsthm}
\usepackage{amssymb}
\usepackage{physics}
\usepackage{stmaryrd}
\usepackage{pgfplots}
\usepackage{graphics, graphicx}
\usepackage[shortlabels]{enumitem}
\usepackage{mathtools}
\usepackage{fullpage}
\usepackage{cancel}
\usepackage{longtable}
\usepackage[colorinlistoftodos]{todonotes}
\usepackage{stmaryrd}
\usepackage{mathscinet}

\usepackage{tikz}
\pgfplotsset{compat=1.11}
\usetikzlibrary{cd}

\newtheorem{theorem}{Theorem}[section]
\newtheorem*{theorem*}{Theorem}
\newtheorem{prop}[theorem]{Proposition}
\newtheorem*{prop*}{Proposition}
\newtheorem{lemma}[theorem]{Lemma}
\newtheorem*{lemma*}{Lemma}

\newtheorem*{obs*}{Observation}
\newtheorem{claim}[theorem]{Claim}
\newtheorem*{claim*}{Claim}
\newtheorem{corollary}[theorem]{Corollary}
\newtheorem*{corollary*}{Corollary}
\theoremstyle{definition}
\newtheorem{definition}[theorem]{Definition}
\theoremstyle{remark}
\newtheorem{remark}[theorem]{Remark}
\newtheorem{example}[theorem]{Example}

\newcommand{\Z}{\mathbb{Z}}
\newcommand{\Q}{\mathbb{Q}}
\newcommand{\R}{\mathbb{R}}
\newcommand{\C}{\mathbb{C}}
\newcommand{\Pro}{\mathbb{P}}
\newcommand{\Aff}{\mathbb{A}}

\newcommand{\ol}[1]{\overline{#1}}

\newcommand{\wt}[1]{\widetilde{#1}}
\newcommand{\wh}[1]{\widehat{#1}}

\newcommand{\sprod}[2]{\left\langle #1,\,#2 \right\rangle}
\newcommand{\definedby}{\coloneqq}

\newcommand{\Coker}{\mathrm{Coker}\;}

\newcommand{\Proj}{\mathrm{Proj}}

\newcommand{\Hom}[3]{\mathrm{Hom}_{#1}\qty(#2,\,#3)}

\newcommand{\Ext}[4][]{\mathrm{Ext}^{#2}_{#1}\qty(#3,\,#4)}

\DeclareMathOperator{\Cone}{cone}
\DeclareMathOperator{\Conv}{conv}
\DeclareMathOperator{\Affine}{aff}
\DeclareMathOperator{\relint}{relint}

\DeclareMathOperator{\Spec}{Spec}

\title{On maximal rational polyhedral fans}
\author{Dan Edidin and Dillon Lisk}

\begin{document}

\maketitle
\begin{abstract}
    In this paper we study the geometry and combinatorics of the possible rational polyhedral fans with a given set of rays. 
    The main questions we consider are when such fans are projective, complete, or simplicial. To answer these questions we use techniques of invariant theory to characterize the quotients of maximal saturated open sets in torus representations.
\end{abstract}
\section{Introduction}
The purpose of this paper is to study the geometry and combinatorics of the possible rational polyhedral fans with a given set of rays.  Since there is a one-to-one correspondence between such fans in $\R^n$ and $n$-dimensional toric varieties, this is equivalent to understanding the geometry of all the toric varieties in a certain family corresponding to the choice of rays.  To characterize the toric varieties whose fans have the same rays, we recall the construction, due to Cox, that realizes $n$-dimensional toric varieties as quotients of an open subset of $\Aff^s$ by a diagonalizable group of rank $s-n$.  Fixing the rays in the fan corresponds to fixing the diagonalizable group  action and allowing the open subset to vary.  This allows us to use equivariant methods to study this problem.

Key properties that we focus on are whether a fan is
\begin{itemize}
    \item \textit{complete}: the union of all its cones is all of $\R^n$,
    \item \textit{projective}: the fan is the normal fan of a rational polytope, and
    \item \textit{simplicial}: every cone $\sigma$ is generated by exactly $\dim{\sigma}$ rays,
\end{itemize}
since these relate directly to the geometry of the corresponding toric variety.

It follows from the definitions that every projective fan is complete, and for a fan to be complete it is clearly necessary 
that every vector in $\R^n$ is a conical linear combination 
of the rays $\rho_1, \ldots , \rho_s$; i.e.  $\Cone(\rho_1, \dots, \rho_s) = \R^n$.  We will thus restrict to the case where this condition holds on the rays, and we will further restrict to fans that are maximal with respect to inclusion, among the fans on our fixed set of rays.  Since inclusions $\Sigma \subset \Sigma'$ correspond to open inclusions between the corresponding toric varieties, to understand the geometry of all fans it suffices to understand the geometry of the maximal ones.  More precisely:
\begin{definition}\label{definition-maximal}
Given a set $R$ of rays, a fan $\Sigma$ is called \textit{maximal with respect to $R$} if it is maximal, with respect to set inclusion, among all fans $\Sigma'$ such that $\Sigma'(1) \subseteq R$.  Note that we do not require equality, so not every ray in $R$ need occur in $\Sigma$.
\end{definition}

In low dimensions, these questions are very easy to answer.  If $n = 1$, there is only one suitable collection of rays -- the two half-lines -- and only one maximal fan, consisting of both rays and $\{0\}$.  This corresponds to the toric variety $\Pro^1$  which is projective and simplicial.  In dimension 2, every spanning set of rays has a unique maximal fan, and that fan is likewise projective and simplicial.

However, the story changes in dimension 3, as there are well-known examples of complete, simplicial, and non-projective toric varieties (see, for example, \cite[p. 71]{fulton1993toric}), and our results will show even more interesting behavior in dimensions $\ge 4$.


\begin{theorem}\label{all-maximal-fans-complete}
In every dimension $n \ge 3$, there exists a set $R$ of rays in $\R^n$ such that every fan in $\R^n$ maximal with respect to $R$ is complete, with both projective and non-projective fans occurring.  In both the projective and non-projective cases, both simplicial and non-simplicial fans occur.
\end{theorem}
\begin{theorem}\label{all-maximal-nonprojective-fans-noncomplete}
In every dimension $n \ge 4$, there exists a set $R$ of rays in $\R^n$ such that every fan in $\R^n$ maximal with respect to $R$ is either projective or non-complete, with both possibilities occurring.  In both the projective and non-projective cases, both simplicial and non-simplicial fans occur.
\end{theorem}
Thus, while there are special examples where these three concepts coincide, there are, in higher dimensions, examples where the three concepts behave quite differently.  The gap between the dimension bounds in these two theorems suggests that everything behaves as well as one can hope in dimension three. To that end we prove the following theorem:
\begin{theorem}\label{3d-fans-complete}
  Let $\Sigma$ be a maximal fan in $\R^3$.  If $\ol{\R^3\setminus\mathrm{supp}\,\Sigma}$ is
  contained in an open half-space then $\Sigma$ is complete.
\end{theorem}
We believe that the condition ``$\ol{\R^3\setminus\mathrm{supp}\,\Sigma}$ is contained in some open half-space'' can be removed, but we have not yet been able to prove this.  Even so, our key example for Theorem \ref{all-maximal-nonprojective-fans-noncomplete} satisfies the condition $\ol{\R^3\setminus\mathrm{supp}\,\Sigma}$, so Theorem \ref{3d-fans-complete} cannot extend to higher dimensions.

It is also reasonable to suspect that things behave better if the number $s$ of rays is close to the dimension $n$ -- i.e., 
the rank of the class group of the associated toric variety is small. In that direction we have the following theorem.
\begin{theorem}\label{rank2-tori-projective}
Any complete toric variety $X$ of dimension $d$ whose fan has $d+2$ rays is projective.
\end{theorem}
This was already proven  in the case that the fan is simplicial
\cite{kleinschmidt1991smooth,acampo2003orbit}, and we prove the non-simplicial case here.

The smallest example we know of a complete, non-projective toric variety has 6 rays in $\R^3$, so the bound $d + 2$ is the best bound of the form $d + i$ in theorem \ref{rank2-tori-projective}.  However, our example for Theorem \ref{all-maximal-nonprojective-fans-noncomplete} is based on a set of 7 rays in $\R^4$, all of whose maximal complete fans is projective, so there may be less restrictive optimal bounds for $s$ in higher dimensions.

Our method for proving Theorems \ref{all-maximal-fans-complete} and \ref{all-maximal-nonprojective-fans-noncomplete} is to use the Cox construction to reduce to enumerating all the subsets of $\Aff^s$ having a good quotient with respect to the action of a fixed diagonalizable group.  Theorems \ref{3d-fans-complete} and \ref{rank2-tori-projective} rely on various results on convexity.

\section{Background}

\subsection{Quotients in Algebraic Geometry}
In this section we discuss quotients by reductive groups for schemes defined over an algebraically closed field. For simplicity of exposition, we assume that our groups are also {\em linearly reductive}. In characteristic 0, any reductive group (i.e., a group having trivial unipotent radical) is linearly reductive. In characteristic $p$, a connected algebraic group is linearly reductive if and only if it is a torus. Since actions of tori are the main focus of this paper, this is not a restricting assumption.

\begin{definition}\label{def.goodquotient}
If a linearly reductive group $G$ acts on a scheme $X$, then a  $G$-invariant morphism $\pi: X \to Y$ is a {\em good quotient} if 
\begin{enumerate}
    \item $\pi$ is affine, and
    \item The natural map ${\mathcal O}_Y \to \pi_*({\mathcal O_X})^G$ is an isomorphism.
\end{enumerate}
In this case, we write $Y = X/G$.    
\end{definition}
\begin{remark}
    Definition \ref{def.goodquotient} is equivalent to the requirement that $Y$ be the good moduli space of the quotient stack $[X/G]$.
    By \cite[Theorem 6.6 ]{alper2012good}, a good quotient is universal for 
    $G$-invariant maps to algebraic spaces, i.e.,
    whenever $\varphi: X \to Z$ is $G$-invariant, there exists a unique morphism $\wt{\varphi}: Y \to Z$ such that $\varphi = \wt{\varphi}\circ\pi$.  In particular, good quotients are unique up to unique isomorphism whenever they exist.
\end{remark}
\begin{remark}
Following \cite{bialynicki-birula1996open, @bialynickibirula1998recipe} we will require that if $X$ is separated then the good quotient $Y$ is also also separated. In this paper we will be considering quotients of open sets in $\Aff^n$, so we necessarily assume that our good quotients are separated schemes and -- since $\Aff^n$ is integral -- in fact, varieties.  There are simple examples of non-separated good quotients, however: consider the action of $\mathbb{G}_m$ on $\Aff^2$ given by the weights $1, -1$, and let $U = \Aff^2 \setminus \{0\}$.  Then the good quotient of $U$ by $\mathbb{G}_m$ is the affine line with two origins, which is not separated.
\end{remark}

Even when good quotients exist, there need not be a bijective correspondence between orbits and points of the quotients. Accordingly we make the following defition:
\begin{definition}
    A $G$-invariant map $\pi: X \to Y$ is a \textit{geometric quotient} if it is a good quotient and, for each algebraically closed field $K$, the $K$-valued points of $Y$ are in one-to-one correspondence with the $G$ orbits of the $K$-valued points of $X$.
\end{definition}


One celebrated source of good quotients, due to Mumford \cite{GIT}, is geometric invariant theory, which we briefly sketch.  
If $L \to X$ is a line bundle on a scheme $X$ with a $G$-action, a \textit{$G$-linearization} of $L$ is a $G$-action on $L$ such
that the projection $L \to X$ is $G$-equivariant (see \cite[Definition 1.6]{GIT} for an equivalent definition in terms of invertible sheaves).  In particular, a $G$-linearization of $L$ induces an action on the global sections of $L$, so we can speak of $G$-invariant sections of $L$.


If $X$ is equipped with a $G$-linearized line bundle $L$, then we say a point $x \in X$ is
\begin{enumerate}
    \item \textit{Semistable} if there is a $G$-invariant global section $s$ of $L^n$ for some $n$ such that $X_s$ is affine and $x \in X_s$, and
    \item \textit{Stable} if, moreover, the action of $G$ on $X_s$ is closed--that is, the orbits of geometric closed points are closed.
\end{enumerate}
We let $X^{ss}$ and $X^s$ denote the open subsets of semistable (respectively stable) points of $X$ (which depend on both $L$ and the choice of linearization).  Then the main result of geometric invariant theory is that, when $G$ is a reductive group, a good quotient $\pi: X^{ss} \to X^{ss}/G$ exists, $X^{ss}/G$ is quasi-projective, and $\pi|_{X^s}: X^s \to X^s/G$ is a geometric quotient.  One writes $X \sslash G$ for $X^{ss}/G$.  If, furthermore, $X$ is projective and $L$ is ample, then $X\sslash G$ is projective.  However, it is known that not all good quotients are quasiprojective, so, in particular, not all good quotients arise via this construction.

We are, therefore, interested in classifying all the ``essentially distinct'' (separated) good quotients $U/G$ that appear as $U$ ranges over the open subsets of $X$.  Now, supposing $U$ is an open subset of $X$ and that $U$ has a good quotient $U/G$ by $G$, we may produce more open subsets with good quotients by considering open sets of the form $\pi^{-1}(V)$  where $V \subset U/G$ is open.  We wish to restrict attention to the open subsets $U$ that \textit{cannot} be obtained in this trivial manner, and accordingly we recall that, given $G$-invariant open subsets $U \subseteq U' \subseteq X$, $U$ is said to be \textit{saturated} in $U'$ if, for every $x \in U$, the closure of the orbit $Gx$ remains closed in $U'$.  
\begin{definition}\cite[Definition 1.2]{bialynicki-birula1996open}
We call a $G$-invariant open set $U \subseteq X$ 
{\em $G$-maximal} (or {\em maximal} if $G$ is understood)
if a separated good quotient $U/G$ exists and $U$ is maximal among such open subsets with respect to the ``saturated in'' relation. \end{definition} 
It follows from the definition of good quotient that any open set of the form  $\pi^{-1}(V)$ is saturated in $U$ for all $V \subseteq P/G$, so we shall primarily restrict to studying the $G$-maximal open subsets of $X$.
\begin{remark}
If $X$ is irreducible and $U$ is a $G$-invariant open set such that
the quotient $U/G$ is complete then 
$U$ is necessarily $G$-maximal, but as we show in this paper
the converse need not hold. Also note that it is possible to have $G$-maximal open sets
$U$ and $U'$ with $U'$ a proper subset $U$ as we show in 
Example \ref{ex.reichstein}.
\end{remark}
\subsection{Toric Varieties}

The \textit{$n$-dimensional (complex) torus} is the group scheme
$\mathbb{G}_m^n$, where $\mathbb{G}_m = \Spec\;\C[t, t^{-1}]$ is the group of
units in $\C$ under multiplication.  If $T = \mathbb{G}_m^n$ is such a torus,
let $N$ denote the group of one-parameter subgroups of $T$, i.e., the set of all
group homomorphisms $\mathbb{G}_m \to T$ under pointwise multiplication, and let
$M$ denote the group of characters of $T$, i.e., group homomorphisms $T \to
\mathbb{G}_m$.  It is a standard fact that $N$ and $M$ are free abelian of rank
$n$, and that the pairing $N \times M \to \Z$ defined by $\langle \lambda, \chi
\rangle = e$ if $\chi \circ \lambda: z \mapsto z^e$ is a perfect pairing.

A \textit{toric variety} is an integral, normal, separated scheme of finite type
over $\Spec\C$ containing a torus $T$ as a dense, open subset, and equipped
with an action of $T$ that extends the natural action of $T$ on itself by
multiplication.

It is well known that there is a one-to-one correspondence between $n$-dimensional toric varieties and fans in $\R^n$.  One way of constructing the toric variety $X_\Sigma$ of a fan $\Sigma$ is to glue together the affine toric varieties $X_\sigma = \Spec(\C[\sigma^\vee \cap M])$ as $\sigma$ ranges over the cones of $\Sigma$.  An alternative construction, due to Cox, constructs the toric variety as a quotient in the following way.  For simplicity, we will assume that the support of $\Sigma$ is full dimensional; the general case simply involves crossing the result here with a torus of appropriate rank.  For each ray $\rho \in \Sigma(1)$, let $u_\rho$ be a generator of $\rho \cap N$.  We have a natural map $\beta: \Z^{\Sigma(1)} \to N$ given by $e_\rho \mapsto
u_\rho$, where $\{e_\rho\}$ is the set of basis vectors.  Then, since $\Sigma$ is full-dimensional, the $u_\rho$ span $N_\R$ and
$\Coker(\beta)$ is finite.  Dualizing in the category of abelian groups gives a short exact sequence
\begin{gather*}
  0 \to M \to \Z^{\Sigma(1)} \to \mathrm{Cl}(X) \to 0,
\end{gather*}
where $\mathrm{Cl}(X) = \Coker(\beta^*)$ will end up being the divisor class group of the toric variety.  Further applying the functor
$D(\text{---}) \definedby \Hom{}{\text{---}}{\mathbb{G}_m}$
gives a short exact sequence of diagonalizable groups
\begin{gather*}
  1 \to G \to T' \to T \to 1,
\end{gather*}
where $G = D(\mathrm{Cl}(X))$ and $T'$ is the torus of $\Aff^{\Sigma(1)}$, which
has character group $\Z^{\Sigma(1)}$.  Thus, $G$ acts on $\Aff^{\Sigma(1)}$ by
restricting the torus action to $G$.

For each cone $\sigma \in \Sigma$, let $x_{\hat\sigma} = \prod_{\rho \not\le
  \sigma}x_\rho$, and let $Z(\Sigma) = V(x_{\hat\sigma}: \sigma \in \Sigma)
\subseteq \Aff^{\Sigma(1)}$.  Then the action of $G$ on $\Aff^{\Sigma(1)}$ restricts to the open subset $U =
\Aff^{\Sigma(1)} \setminus Z(\Sigma)$, and we have that $X_\Sigma = U/G$.

We can characterize the fan in terms of the intrinsic toric variety
structure as follows. Recalling that $\mathbb{G}_m = \Spec\,\C[t, t^{-1}] \subseteq \Spec\,\C[t] =
\Aff^1$, we say that a one-parameter subgroup $\lambda \in N$ \textit{has a
  limit as $t \to 0$} if it extends to a map $\Aff^1 \to X_\sigma$.  Since
$X_\sigma$ is separated, this extension is unique provided it exists.  We write
$\lim_{t \to 0}\lambda(t) \in X_\sigma$ for the image of the closed point $(t
= 0) \in \Aff^1$ in this case.  Then we have the following: each $X_\sigma$ has a distinguished closed point $\gamma_\sigma$ such that, for every $\lambda
\in N$,
\begin{enumerate}
\item
  $u \in \sigma$ if and only if $\lim_{t \to 0}\lambda(t)$ exists, and

\item
  $u \in \relint(\sigma)$ if and only if $\lim_{t \to 0}\lambda(t) =
  \gamma_\sigma$.
\end{enumerate}
This allows
us to recover the fan $\Sigma$ from $X$ as $\Sigma = \{\sigma_\lambda: \lambda \in N,\; \lim_{t \to 0}\lambda(t) \text{ exists}\}$, where $\sigma_\lambda = \ol{\Cone\{\mu \in N: \lim_{t \to     0}\mu(t) = \lim_{t \to 0}\lambda(t)\}}$

There is another construction of toric varieties -- specifically projective toric varieties -- that we will need to make use of.  Suppose that $P \subseteq M_\R$ is a rational polytope in the character lattice $M$ of a torus $T$.  Then the set of
lattice points $P \cap M$ defines a map from $T \to \Aff^{s+1}$ where $s+1=|P \cap M|$. Moreover, since these are non-vanishing on
the torus there is an induced morphism $\varphi_P \colon T \to \Pro^s$.  It can be shown that for some value of $k$ the map $\varphi_{kP}$ defined by the lattice points in $kP$ is an embedding. We then define the projective toric variety $X_P$ to be the closure of $\varphi(T)$ in $\Pro^s$.  For $k$ sufficiently large this toric variety can be proven to be independent of $k$, and is equal to $X_\Sigma$, where $\Sigma$ is the normal fan of $P$ in $N$.  In particular, $X_P = X_{P + \mathbf{a}}$ for every $\mathbf{a} \in M_\R$.

\subsection{Toric GIT}

Consider the affine space $\Aff^s$, a toric variety with torus $S=\mathbb{G}_m^s$, and let $G$ be a subgroup of $S$.  Then $G$ acts  diagonally on $\Aff^s$, and we can specialize geometric invariant theory to this situation.  Every line bundle over affine space is trivial, and $G$-linearizations of the trivial bundle $\Aff^n\times\Aff^1 \to \Aff^s$ are all of the form $g(x,\lambda) = (gx, \chi(g)\lambda)$ for some character $\chi$ of $G$, so henceforth we will identify characters with $G$-linearizations without comment.

Just as in the Cox construction, the short exact sequence of groups
\begin{gather*}
    0 \to G \to S \to \wt{T} \to 0,
\end{gather*}
where $\wt{T} = S/G$, gives, upon taking characters, a short exact sequence of finitely-generated abelian groups
\begin{gather*}
    0 \to M \overset{\delta}{\to} \Z^s \overset{\gamma}{\to} \wh{G} \to 0,
\end{gather*}
where $\wh{G}$ is the character group of $G$ and $M$ the character lattice of $\wt{T}$.  The action of $G$ on $\Aff^s$ is entirely determined by the characters $\chi_i \definedby \gamma(e_i)$, where $e_i$ are the standard basis vectors in $\Z^s$.  Dualizing this, we get
\begin{gather*}
    0 \to \wh{G}^* \overset{\gamma^*}{\to} \Z^s \overset{\delta^*}{\to} N \to \Ext{1}{\wh{G}}{\Z},
\end{gather*}
which gives us a distinguished set $\nu_i = \delta^*(e_i)$ of element in $N$, the lattice of one-parameter subgroups of the torus $\wt{T}$.  We observe that $\Ext{1}{\wh{G}}{\Z}$ is torsion, so that, tensoring both sequences with $\R$, we obtain
\begin{gather*}
    0 \to M_\R \overset{\delta}{\to} \R^s \overset{\gamma}{\to} \wh{G}_\R \to 0,
\end{gather*}
and
\begin{gather*}
    0 \to \wh{G}^*_\R \overset{\gamma^*}{\to} \R^s \overset{\delta^*}{\to} N_\R \to 0.
\end{gather*}
We write $\beta_i$ for the image of $\chi_i$ in $\wh{G}_\R$ and note that, if $G$ is a torus, the map $\wh{G} \to \wt{G}_\R$ is injective, so we can identify $\chi_i$ with $\beta_i$.  Since $\wt{T}$ is always a torus, we will simply write $\nu_i$ for the image of $\nu_i$ in $N_\R$.  The sets $\{\beta_i\} \subseteq \wh{G}_\R$ and $\{\nu_i\} \subseteq N_\R$ obtained in this way are said to be \textit{Gale dual} to each other.  

Now take any $\chi \in \wh{G}$.  Let $\Z^s_{\ge 0} = \{(a_1,
\dots, a_s) \in \Z^s: a_i \ge 0\;\forall i\}$ be the positive orthant
and $\mathbf{a} = (a_1, \dots, a_s) \in \Z^s$ be any lift of $\chi$,
we have that $(\gamma^{-1}(\chi) \cap \Z^s_{\ge 0})-\mathbf{a}$ is in
the kernel of $\gamma$, hence in the image of $\delta$, so that
$\delta^{-1}((\gamma^{-1}(\chi) \cap \Z^s_{\ge 0})-\mathbf{a})
\subseteq M$ is well-defined.  It can be shown that the GIT quotient
corresponding to the $G$-linearization induced by $\chi$, denoted
$\Aff^s \sslash_\chi G$, is the toric variety associated to the
polyhedron $P_\mathbf{a} \definedby
\Conv(\delta^{-1}((\gamma^{-1}(\chi) \cap \Z^s_{\ge 0})-\mathbf{a}))
\subseteq M_\R$.  Note that, although $P_\mathbf{a}$ depends on the
choice of $\mathbf{a}$, the different $P_\mathbf{a}$ will be
translates of each other and so induce the same toric variety.

In terms of the Gale dual elements $\nu_i$, $P_\mathbf{a}$ may be more simply expressed as $\{m \in M_\R: \sprod{m}{\nu_i} \ge -a_i \;\forall i\}$.  This lets us characterize the $\chi$-semistable points directly in terms of the polytope $P$ as follows: define the \textit{$i$-th virtual facet} $F_i$ of $P$ as $F_i = \{m \in P: \sprod{m}{\nu_i} = -a_i\}$, and define the ideal $B = (\prod_{v \notin F_i}x_i: v \text{ is a vertex of }P)$.  Then $X^{ss} = \Aff^s \setminus Z(B)$.

The \textit{secondary fan} is a fan in $M$ that characterizes how the quotients $\Aff^s\sslash_\chi G$ vary as $\chi$ varies in $M$.  Take the cone $\Cone(\beta_1, \dots, \beta_s)$ and divide it along all the subcones $\Cone(\beta_{i_1}, \dots, \beta{i_t})$, $\{i_1, \dots, i_t\} \subseteq \{1, \dots, s\}$, such that the dimension of $\Cone(\beta_{i_1}, \dots, \beta{i_t})$ is strictly less than the dimension of $\Cone(\beta_1, \dots, \beta_s)$.  Though it is non-trivial to prove, this results in a fan, and we have the following results:
\begin{enumerate}
    \item $\Aff^s\sslash_\chi G$ is constant as $\chi$ varies over the relative interior of each cone of the secondary fan.
    \item $\Aff^s\sslash_\chi G = \emptyset$ for $\chi$ not in the support of the secondary fan.
    \item There are no $\chi$-stable points iff $\chi$ is on the boundary of the secondary fan, and both of these conditions imply that $\Aff^s\sslash_\chi G$ has dimension strictly less than $s - \dim G$. We call the quotient \textit{degenerate} if either of these equivalent conditions holds.
\end{enumerate}

\begin{remark}\label{rays-of-nd-GIT-quotients}
When $\chi$ is in the relative interior of the secondary fan, the corresponding GIT quotient will have torus $\wt{T}$ and have a fan whose rays are generated by some (but not necessarily all) of the $\nu_i$s.
\end{remark}

\subsection{Bia{\l}ynicki-Birula and \'{S}wi\polhk ecicka's Characterization of Good Quotients}\label{BBS}

In order to enumerate all the maximal open subsets with good
quotients, we utilize the results of
Bia{\l}ynicki-Birula and \'{S}wi\polhk ecicka
\cite{bialynicki-birula1996open}.  The relevant parts of
their theory are as follows.

Let $T$ be a torus with character lattice $M$ acting diagonally on $X = \Aff^n$
with weights $\chi_1, \dots, \chi_n$.  The convex polytopes of the form
$\mathrm{conv}(0, \chi_{i_1}, \dots, \chi_{i_s}) \subseteq M_\R$ where $\{i_1,
\dots, i_s\} \subseteq \{1, \dots, n\}$ are called \textit{distinguished
  polytopes}.  To each point $x = (x_1, \dots, x_n) \in \Aff^n$, we can
associate a distinguished polytope $P(x)$ via
\begin{gather*}
  P(x) = \mathrm{conv}(0, \chi_i: x_i \ne 0)
\end{gather*}
Moreover, each distinguished polytope has an associated subset
\begin{gather*}
  X(P) = \{x \in \Aff^n: P(x) = P\},
\end{gather*}
and it follows that $(x_1, \dots, x_n) \in X(P)$ if and only if $x_i \ne 0$ for
every vertex $\chi_i$ of $P$ and $x_j = 0$ for every $\chi_j \notin P$ (note
that there may be characters contained in $P$ that are not vertices of $P$).
Given a subset $Y \subseteq \Aff^n$, we set
\begin{gather*}
  \Pi(Y) = \{P(y): y \in Y\},
\end{gather*}
and conversely, if $\Pi$ is a collection of distinguished polytopes, then we set
\begin{gather*}
  X(\Pi) = \bigcup_{P \in \Pi} X(P).
\end{gather*}
A subset $Y$ of $\Aff^n$ is called \textit{combinatorially closed} if $Y =
X(\Pi(Y))$.  The relevant results of \cite{@bialynickibirula1998recipe} are:

\begin{theorem} \cite[Theorem 1.6]{@bialynickibirula1998recipe}
If $Y$ is a $T$-maximal subset of $\Aff^n$, then $Y$ is combinatorially closed.
\end{theorem}

\begin{theorem} \cite[Theorem 1.4]{@bialynickibirula1998recipe}\label{BBS-good-quotient-characterization} 
$X(\Pi)$ is open in $\Aff^n$ and has a good quotient if and only if the
  following two conditions hold:
\begin{enumerate}[(i)]
\item
  $\Pi$ is upwardly closed, that is, if $P \in \Pi$ and if $Q \supseteq P$ is
  another distinguished polytope, then $Q \in \Pi$.

\item
  If $P, Q \in \Pi$ and $P \cap Q$ is a face of $P$ or $Q$, then $P \cap Q \in
  \Pi$.
\end{enumerate}
In fact, assuming (i) holds, (ii) is equivalent to ``If $P, Q \in \Pi$ and $P
\cap Q$ is \textit{contained in} a face of $P$ or $Q$, then $P \cap Q \in
\Pi$.''
\end{theorem}

\begin{theorem}\label{saturation}
Given two collections of distinguished polytopes $\Pi_1 \subseteq \Pi$,
$X(\Pi_1)$ is saturated in $X(\Pi)$ if and only if whenever $P \in \Pi_1$ and
some face $F$ of $P$ is contained in $\Pi$, we have $F \in \Pi_1$.
\end{theorem}

\begin{remark}\label{rays-of-nd-quotients}
Let $S$ be the big torus in $\Aff^n$ so that $T$ is a subgroup of $S$.  By \cite{@bialynickibirula1998recipe}, every \textit{non-degenerate} good quotient of $\Aff^n$ by $T$ -- that is, a quotient on which $S/T$ acts effectively -- is a toric variety all of whose rays are generated by the Gale duals of the characters $\chi_1, \dots, \chi_n$, just as in the GIT case.
\end{remark}

\section{Enumerating Maximal Open Sets}

Let $N$ be a lattice, $\nu_1, \dots, \nu_s \in N$, and $\rho_i$ be the ray generated by $\nu_i$.  By the Cox construction, this gives us an action of a diagonalizable group $G$ on $\Aff^s$ defined by the characters $\chi_1, \dots, \chi_s$ Gale dual to the $\nu_i$. By Remark \ref{rays-of-nd-GIT-quotients} 
every non-degenerate $G$-maximal open subset $U$ of $\Aff^s$ gives a toric variety $U/G$ whose fan is in $N$ and whose rays are a (possibly proper) subset of $\rho_1, \dots, \rho_s$.  We claim that the converse is also true; our argument is a slight generalization of the Cox construction.
\begin{prop} \label{prop.maximal}
Let $\Sigma$ be a maximal fan whose rays are contained in $\rho_1, \ldots , \rho_s$ then $X_\Sigma$ is the quotient 
of a maximal $G$-subset of $\Aff^s$
\end{prop}

The key step in the proof is the following lemma:

\begin{lemma}
  Let $N, N'$ be  lattices and let $\varphi: N' \to N$ be a map with finite cokernel.  Suppose $\sigma'$ is a strictly-convex rational polyhedral cone in $N'_\R$ and $\sigma$ such a cone in $N_\R$ such that $\varphi_\R(\sigma') = \sigma$.  Let $M = N^*$ and $M' = (N')^*$.  Then the ring map $\psi: \C[M \cap \sigma^\vee] \to \C[M' \cap (\sigma')^\vee]$ dual to the induced map of toric varieties $X_{\sigma'} \to X_\sigma$ is isomorphic to the inclusion $\C[M' \cap (\sigma')^\vee]^G \subseteq \C[M' \cap (\sigma')^\vee]$, where $G$ is the diagonalizable group with character group $\Coker{\varphi^*}$.
\end{lemma}
\begin{proof}
  As an additive group, $\C[M \cap \sigma^\vee]$ is the direct sum of $\C x^m$ for $m \in M \cap \sigma^\vee$, and $\psi$ sends $x^m$ to $x^{\varphi^*(m)}$.  Notice that, since $\varphi^*$ is an injection, so is $\psi$, and observe that, for $m \in M$,
  \begin{align*}
      &m \in \sigma^\vee \\
      \iff &\sprod{m}{n} \ge 0 \;\forall n \in \sigma \\
      \iff &\sprod{m}{\varphi(n')} \ge 0 \;\forall n' \in \sigma' \\
      \iff &\sprod{\varphi^*(m)}{n'} \ge 0 \;\forall n' \in \sigma' \\
      \iff &\varphi^*(m) \in (\sigma')^\vee
  \end{align*}
  Since the $\wh{G} = \Coker{\varphi^*}$, we get a natural induced action of $G$ on $\C[M' \cap (\sigma')^\vee]$ via
  \begin{gather*}
      g\cdot x^{m'} = \ol{m'}(g)x^{m'},
  \end{gather*}
  where $\ol{m'}$ is the image of $m'$ in $\wh{G}$.  Since this action preserves the natural direct sum decomposition of $\C[M' \cap (\sigma')^\vee]$, it suffices to check the monomials in the image of $\C[M \cap \sigma^\vee]$ are exactly the $G$-invariant monomials of $\C[M' \cap (\sigma')^\vee]$.  But this is true since $g\cdot x^{m'} = x^{m'}$ for all $g \in G$ if and only if $\ol{m'}$ is the trivial character on $G$ if and only if $m' = \varphi^*(m)$ for some $m \in M$, and we have already verified that in this case $m' \in (\sigma')^\vee$ if and only if $m \in \sigma$.
\end{proof}
\begin{proof}[Proof of Proposition \ref{prop.maximal}]
Now suppose that $\Cone(\nu_1, \dots, \nu_s) = N_\R$ and let $\Sigma$ be a fan maximal with respect to $\rho_1, \dots, \rho_s$.  Let $\varphi: \Z^s \to N$ be the map sending $e_i \mapsto \nu_i$, where $e_i$ is a standard basis vector of $\Z^s$.  We can ``lift'' this to a fan on $\Z^s$ in, in general, several different ways, of which we will list the two extreme cases.  For every cone $\sigma \in \Sigma$, define cones
\begin{gather*}
    \wt{\sigma} = \Cone(e_i: \rho_i \text{ is a ray of } \sigma) \\
    \wh{\sigma} = \Cone(e_i: \rho_i \subseteq \sigma)
\end{gather*}
and let $\wt{\Sigma} = \{\wt{\sigma}: \sigma \in \Sigma\}$, $\wh{\Sigma} = \{\tau: \tau \text{ is a face of }\wh{\sigma},\, \sigma \in \Sigma\}$.  If the rays of $\Sigma$ do not include every $\rho_i$, then these two fans will be distinct.  Both $\wt{\Sigma}$ and $\wh{\Sigma}$ are fans in $\Z^s$ with the following properties:
\begin{enumerate}
    \item $\wt{\sigma}$ and its faces are the only cones of $\wt{\Sigma}$ that $\varphi$ maps into $\sigma$, and
    \item $\varphi(\wt{\sigma}) = \sigma$,
\end{enumerate}
and similarly for $\wh{\Sigma}$.  From these it follows that there are induced maps $f: X_{\wt{\Sigma}} \to X_\Sigma$ and $g: X_{\wh{\Sigma}} \to X_\Sigma$ which are affine and satisfy $f^{-1}(X_\sigma) = X_{\wt{\sigma}}$, and similarly for $g$.  Since $\Cone(\nu_1, \dots, \nu_s) = N_\R$, it follows that the cokernel of $\varphi$ is finite, so the lemma applies in this case to show that each restriction $f|_{X_{\wt{\sigma}}}: X_{\wt{\sigma}}: X_\sigma$ is $G$-invariant and satisfies the condition that the natural map $\mathcal{O}_{X_\sigma} \to f_*\mathcal{O}_{X_{\wt{\sigma}}}^G$ is an isomorphism.  Thus, $f$ is a good quotient, and similarly for $g$.  Just as in the Cox construction, $X_{\wt{\sigma}}$ and $X_{\wh{\sigma}}$ are naturally embedded in $\Aff^s$ as open subsets $\wt{U}$ and $\wh{U}$, showing that $X_\Sigma$ indeed arises as the good quotient of at least one $G$-invariant open subset.  If these open subsets were contained as proper saturated open sets in a larger
$G$-invariant open $U \subset \Aff^s$, then we would be able to embed $X_\Sigma$ as an open subset in the toric variety $U/G$.
The toric variety $U/G$ would also have a fan with rays among $\rho_1, \dots, \rho_s$, but would necessarily contain $\Sigma$, contradicting the maximality of $\Sigma$.  Therefore, $\wt{U}$ and $\wh{U}$ must be $G$-maximal. 
\end{proof}

\begin{remark}
Suppose we have an inclusion $U' \subseteq U$ of $G$-maximal subsets of $\Aff^s$ with non-degenerate quotients.  Then both $U'/G$ and $U/G$ are toric varieties with torus $T = S/G$; call the fan of $U'/G$ $\Sigma'$ and the fan of $U/G$ $\Sigma$.  The inclusion $U' \subseteq U$ induces a map $\varphi: U'/G \to U/G$ that is the identity on $T$, and by the characterization of the fans in terms of limit points, it follows that each cone $\sigma'$ of $\Sigma'$ must be entirely contained in a cone $\sigma$ of $\Sigma$.  

If $\Sigma$ and $\Sigma'$ have the same rays and same support, then it follows that each cone $\sigma \in \Sigma$ is the union of possibly several cones in $\Sigma'$.  In other words, the inclusion $U' \subseteq U$ corresponds to subdividing some cones of $\Sigma$ without adding any more rays.  In this case, if $\Sigma$ is simplicial, then $\Sigma = \Sigma'$.

Since any fan can be subdivided to a simplicial fan without adding new rays \cite[Proposition 11.1.7]{CLS}, we always have that any minimal $G$-maximal open set in $\Aff^s$ gives a simplicial fan. However, the converse need not be true, as the following shows.
\end{remark}
\begin{example}\label{ex.reichstein}
Consider the inclusion of open sets $U' = \Aff^7\setminus (Z(c) \cup Z(b, d, f) \cup Z(a, e, g))$ and $U = \Aff^7\setminus (Z(b, d, f) \cup Z(c, d) \cup Z(c, e) \cup Z(a, e, g))$ from the setup of Proposition \ref{dim4-rays} below.  
We have that $U' \subseteq U$, but the induced map on quotients $U'/G \to U/G$ is actually an isomorphism
as both quotients are equal to the simplicial toric variety $X_\Sigma$ where
$\Sigma$ is the fan in $\R^4$ with primitive collection $\{ \{\rho_3\}, \{\rho_2,\rho_4,\rho_6\}, \{\rho_1,\rho_5, \rho_7\}$. Here
the ray $\rho_i$ is generated by the primitive vector $\nu_i$ for $i=1, \ldots , 7$ where $\nu_i$ are given as:
\begin{gather*}
      \begin{array}{lcrrrr}
      \nu_1 &=& (1, & 0, & 0, & 0) \\
      \nu_2 &=& (0, & 1, & 0, & 0) \\
      \nu_3 &=& (0, & 0, & 1, & 0) \\
      \nu_4 &=& (0, & 0, & 0, & 1) \\
      \nu_5 &=& (-2, & 1, & 1, & 1) \\
      \nu_6 &=& (-1, & -1, & 2, & 1) \\
      \nu_7 &=& (2, & -1, & -4, & -3)
      \end{array}
  \end{gather*}

Precisely,
this fan has 9 maximal cones which are each spanned by four rays - two from the set $\{\rho_2, \rho_4, \rho_6\}$ and two from
the set $\{\rho_1, \rho_5, \rho_7\}$. In particular $\Sigma$ is an example of a maximal fan which does not contain the ray $\rho_3$. The reason
is that the ray $\rho_3$ lies in the relative interior of the cone spanned by $\rho_1,\rho_2,\rho_6,\rho_7$. 
Both quotients are GIT quotients. However the
action of $G$ on $U$ has points with positive dimensional stabilizer so $U^{sss} \neq \emptyset$, while
for the action of $G$ on $U'$, $(U')^{s} = (U')^{ss}$.  At the level of stacks, 
the Deligne-Mumford quotient stack $[U'/G]$ is obtained from the stack $[U/G]$ by a divisorial Reichstein transformation \cite[Section 6]{edidin2012partial}
along the Cartier divisor defined by the invariant function $c$.
The induced map of quotients is $U'/G \to U/G$ is a blowup
by \cite[Proposition 3.7]{edidin2021canonical}. Since
the map of quotients is also an isomorphism
it is necessarily the blowup along a Cartier divisor in $U/G$. 
Note that the two open sets $U'$ and $U$ in $\Aff^7$ are exactly the two maximal open sets $\wt{U}$ and $\wh{U}$ associated
to the simplicial toric variety $X_\Sigma$ which were defined in the proof of Proposition \ref{prop.maximal}.
\end{example}

\begin{remark}
Every toric GIT quotient is projective by construction which
implies that the non-empty sets of $\chi$-semistable points are always $G$-maximal. Moreover, in our setup the converse holds as well:
\end{remark}

\begin{prop}\label{Projective-quotients-are-GIT}
If $X = V/G$ for some $G$-maximal open set $V \subseteq \Aff^s$ is projective, then $X = \Aff^s\sslash_LG$ for some $G$-linearized line bundle $L$ on $\Aff^s$.
\end{prop}
\begin{proof}
In the case that the fan of the quotient uses all the rays, this follows from \cite[Proposition 14.1.9]{CLS}; in the general case, we proceed as follows: by the proof of Proposition \ref{prop.maximal}, the quotient $X$ is also equal to $U/G$, where $U$ is the open subset determined by the fan $\wh{\Sigma}$.  Since $\wh{\Sigma}$ uses all the rays generated by the standard basis vectors, the complement of $U$ has codimension at least two in $\Aff^s$.  By \cite[Converse 1.12]{GIT}, there is a $G$-linearized line bundle $L$ \textit{on $U$} such that $U/G = U\sslash_L G$.  Since $\Aff^s$ is smooth and the codimension of $\Aff^s\setminus U$ is at least two, $L$ extends to a $G$-linearized line bundle on $\Aff^s$ which we will still denote by $L$.  The inclusion $U \subseteq \Aff^s$ induces a restriction map $\bigoplus_{i \ge 0}\Gamma(\Aff^s, L^{\otimes i})^G \to \bigoplus_{i \ge 0}\Gamma(U, L^{\otimes i})^G$, which is an isomorphism since the complement of $U$ has codimension at least two.  Since $U\sslash_L G = \Proj(\bigoplus_{i \ge 0}\Gamma(U, L^{\otimes i})^G)$ and $\Aff^s\sslash_L G = \Proj(\bigoplus_{i \ge 0}\Gamma(\Aff^s, L^{\otimes i})^G)$, this shows that $X_\Sigma$ is a GIT quotient of $\Aff^s$ by $G$.
\end{proof}

\section{Proofs of Theorems \ref{all-maximal-fans-complete} and \ref{all-maximal-nonprojective-fans-noncomplete}}

To prove Theorems \ref{all-maximal-fans-complete} and
\ref{all-maximal-nonprojective-fans-noncomplete}, we first exhibit
explicit examples in dimensions 3 and 4 respectively, using these to
bootstrap the result to higher dimensions.  At the level of quotients,
we take the product with a number of copies of $\Pro^1$, proving that
all the notions involved behave well with respect to taking products.
The key technical result we need is that maximal fans on the rays
corresponding to products are all products of maximal fans on the
separate factors.

\subsection{Dimensions 3 and 4}

\begin{prop}\label{dim3-rays}
There exists a set $R$ of six rays in $\R^3$ such that every fan in $\R^3$ maximal with respect to $R$ is complete, with both projective and non-projective fans occurring.  In both the projective and non-projective cases, both simplicial and non-simplicial fans occur.
\end{prop}
\begin{proof}
  Consider the following six vectors in $\R^3$ from which the rays in $R$ will be generated:
  \begin{gather*}
    \begin{array}{rrr}
    (1, & 0, & 0) \\
    (0, & 1, & 0) \\
    (2, & 2, & 3) \\
    (-1, & -2, & -2) \\
    (-2, & -1, & -2) \\
    (0, & 0, & 1)
    \end{array}
\end{gather*}
We will take the Gale dual of these vectors to transform the problem of categorizing the fans on these rays into the problem of categorizing the non-degenerate good quotients of an action on $\Aff^6$.  These rays define a linear map $\R^6 \to \R^3$ with matrix
\begin{gather*}
    \begin{bmatrix}
    1 & 0 & 2 & -1 & -2 & 0 \\
    0 & 1 & 2 & -2 & -1 & 0 \\
    0 & 0 & 3 & -2 & -2 & 1
    \end{bmatrix}
\end{gather*}
One then determines the kernel of this matrix; it can be verified that a basis of the kernel is given by $(-2, -2, 1, 0, 0, -3)$, $(1, 2, 0, 1, 0, 2)$, and $(2, 1, 0, 0, 1, 2)$, i.e., that the kernel is the image of the injection $\R^3 \to \R^6$ with matrix
\begin{gather*}
    \begin{bmatrix}
    -2 & 1 & 2 \\
    -2 & 2 & 1 \\
    1 & 0 & 0 \\
    0 & 1 & 0 \\
    0 & 0 & 1 \\
    -3 & 2 & 2
    \end{bmatrix}
\end{gather*}
Taking the transpose of this matrix gives
\begin{gather*}
    \begin{bmatrix}
    -2 & -2 & 1 & 0 & 0 & -3 \\
    1 & 2 & 0 & 1 & 0 & 2 \\
    2 & 1 & 0 & 0 & 1 & 2
    \end{bmatrix}
\end{gather*}
the columns of which are the vectors Gale dual to the six ray generators.  These six dual vectors can be interpreted as six characters of the torus $T = \mathrm{G}_m^3$, which in turn specifies a diagonal action of $T$ on $\Aff^6$.  This matrix is called the \textit{weight matrix} of this action.

The maximal fans with respect to $R$ are then naturally in one-to-one correspondence with the maximal non-degenerate good quotients of $\Aff^6$ with respect to this action.  Using the results summarized in section \ref{BBS}, we implemented an algorithm in Python + SAGE to enumerate these, the full results being tabulated in the appendix.  In summary, out of 87 non-degenerate quotients, 73 are projective (54 simplicial, 19 not) and 14 are complete, non-projective (2 simplicial, 12 not).
\end{proof}

\begin{prop}\label{dim4-rays}
There exists a set $R$ of seven rays in $\R^4$ such that every fan in $\R^4$ maximal with respect to $R$ is either projective or non-complete, with both possibilities occurring.  In both the projective and non-projective cases, both simplicial and non-simplicial fans occur.
\end{prop}
\begin{proof}
  The proof is similar to the previous.  In this case, the seven rays are generated by
  \begin{gather*}
      \begin{array}{rrrr}
      (1, & 0, & 0, & 0) \\
      (0, & 1, & 0, & 0) \\
      (0, & 0, & 1, & 0) \\
      (0, & 0, & 0, & 1) \\
      (-2, & 1, & 1, & 1) \\
      (-1, & -1, & 2, & 1) \\
      (2, & -1, & -4, & -3)
      \end{array}
  \end{gather*}
  and the corresponding action has weight matrix
  \begin{gather*}
  \begin{bmatrix}
    3 & 0 & -3 & -2 & 1 & 1 & 0 \\
    0 & 3 & -3 & -1 & -1 & 2 & 0 \\
    1 & 1 & 1 & 1 & 1 & 1 & 1
  \end{bmatrix}
\end{gather*}
Of the resulting 112 non-degenerate quotients, 85 are projective (36 simplicial, 49 not), and the remaining 27 are non-complete (8 simplicial, 19 not).
\end{proof}
\subsection{Product fans}
To extend these results to higher dimensions, we need to analyze how fans behave with respect to direct sums of lattices.  Let $N$ and $N'$ be two lattices.  If $\Delta$ is a fan in $N_\R$ and $\Delta'$ is a fan in $N'_\R$ then the \textit{product fan $\Delta \times \Delta'$} is the fan $\{\sigma \times \sigma': \sigma \in \Delta, \sigma' \in \Delta'\}$.  In terms of the corresponding toric varieties, we have that $X_{\Delta \times \Delta'} = X_\Delta \times X_{\Delta'}$.

Now suppose we are given $\nu_1, \dots, \nu_s \in N$ and $\nu'_1, \dots, \nu'_t \in N'$.  Let $\wt{\nu_i}$, $\wt{\nu'_j}$ denote the images of these in $N \oplus N'$ under the two inclusion maps.  It is easy to see that these are the rays of all the product fans $\Delta \times \Delta'$ if $\Delta$ is a fan on $\nu_1, \dots, \nu_s$ and $\Delta'$ is a fan on $\nu'_1, \dots, \nu'_t$.  Our goal in this section is to provide a partial converse; that is, suppose $\Sigma$ is a fan in $N \times N'$ all of whose rays are generated by the $\wt{\nu_i}$ and $\wt{\nu'_j}$.  We will show that, while $\Sigma$ may not be a product fan itself, it can always be obtained from a product fan by removing cones.

\begin{lemma}
Every cone $\sigma \in \Sigma$ is of the form $\tau \times \tau'$, where $\tau$ is a cone in $N$ with rays among the $\nu_i$ and $\tau'$ is a cone in $N'$ with rays among the $\nu'_j$.
\end{lemma}
\begin{proof}
Let $\wt{\nu_i}$, $i \in I$ and $\wt{\nu'_j}$, $j \in J$ be the rays of $\sigma$ and let $\tau = \Cone(\nu_i: i \in I)$, $\tau' = \Cone(\nu'_j: j \in J)$.  Then the elements of $\sigma$ are exactly those of the form
\begin{gather*}
    \sum_{i \in I} \lambda_i\wt{\nu_i}
    + \sum_{j \in J} \lambda'_j\wt{\nu'_j}
    = \qty(\sum_{i \in I} \lambda_i\nu_i,\;
    \sum_{j \in J} \lambda'_j\nu'_j),
\end{gather*}
with $\lambda_i$, $\lambda'_j \ge 0$, which are exactly the elements of $\tau\times\tau'$.
\end{proof}

\begin{lemma}
A subset $F$ of $\tau \times \tau'$ is a face of $\tau \times \tau'$ if and only if $F = \mu \times \mu'$, where $\mu$ is a face of $\tau$ and $\mu'$ is a face of $\tau'$.
\end{lemma}
\begin{proof}
Suppose first that $F$ is a face of $\tau\times\tau'$.  By the previous, every face of $\tau \times \tau'$, being a cone of $\Sigma$ by the fan condition, is of the form $\mu\times\mu'$, so it remains to show that $\mu$ is a face of $\tau$, for the rest will follow by symmetry.  If $\mu\times\mu'$ is empty, one can take $\mu = \mu' = \emptyset$, so suppose $\mu$ and $\mu'$ are both non-empty and suppose that $x, y \in \tau$ are such that $tx + (1-t)y \in \mu$ for some $t \in (0,1)$.  Letting $x'$ be any point of $\mu'$, we have that $t(x,x') + (1-t)(y,x') \in \mu\times\mu'$, so, since this is a face of $\tau\times\tau'$, we have that $(x,x'), (y,x') \in \mu\times\mu'$, implying that $x,y \in \mu$, and $\mu$ is a face of $\tau$.

Conversely, suppose that $\mu$ is a face of $\tau$ and $\mu'$ is a face of $\tau'$.  Let $t(x,x') + (1-t)(y,y') \in \mu\times\mu'$ for some $(x,x'),(y,y') \in \tau\times\tau'$ and some $t \in (0,1)$.  Then $tx + (1-t)y \in \mu$, and thus $x,y \in \mu$.  Similarly we get $y,y' \in \mu'$, so $(x,x'), (y,y') \in \mu\times\mu'$ and $\mu\times\mu'$ is a face of $\tau\times\tau'$.
\end{proof}

\begin{corollary}
  A product cone $\tau \times \tau'$ is simplicial if and only if both $\tau$ and $\tau'$ are.
\end{corollary}
\begin{proof}
  By the previous, the rays of $\tau \times \tau'$ must be of the form $\rho \times \{0\}$, $\rho$ a ray of $\tau$, or $\{0\} \times \rho'$, $\rho'$ a ray of $\tau'$.  Thus, the number of rays of $\tau \times \tau'$ is the sum of the number of rays of $\tau$ and $\tau'$.  Since the dimension of $\tau\times\tau'$ is likewise the sum of the dimensions of $\tau$ and $\tau'$, the result follows.
\end{proof}

\begin{theorem}
The fan $\Sigma$ is a subfan of a product fan.
\end{theorem}
\begin{proof}
Let $\Delta$ be the collection of cones $\pi_N(\sigma)$ and $\Delta'$ be the collection of $\pi_{N'}(\sigma)$ for $\sigma \in \Sigma$.  We have established that every cone of $\Sigma$ is of the form $\tau\times\tau'$ where $\tau \in \Delta$ and $\tau' \in \Delta'$, so it remains to verify that $\Delta$ and $\Delta'$ are actually fans.  Given $\tau \in \Delta$, let $\sigma \in \Sigma$ be such that $\tau = \pi_N(\sigma)$, so that $\sigma = \tau\times\tau'$.  Then, for any face $\mu$ of $\tau$, the previous shows that $\mu\times\tau'$ is a face of $\tau\times\tau'$, so that $\mu\times\tau' \in \Sigma$ and $\mu \in \Delta$.

Now suppose $\tau_1, \tau_2 \in \Delta$ and let $\tau_1', \tau_2'$ be such that $\tau_1\times\tau_1', \tau_2\times\tau_2' \in \Sigma$.  We have that $\tau_1\times\tau_1' \cap \tau_2\times\tau_2' = (\tau_1\cap\tau_2)\times(\tau_1'\times\tau_2')$ is a face of both $\tau_1\times\tau_1'$ and $\tau_2\times\tau_2'$.  Now, since every cone of $\Sigma$ contains the origin, we have that $\tau_1\cap\tau_2$ and $\tau_1'\cap\tau_2'$ are non-empty, and, thus, the previous lemma shows that $\tau_1\cap\tau_2$ is a face of $\tau_1$ and $\tau_2$.  Therefore, $\Delta$ is a fan, and by symmetry so is $\Delta'$.
\end{proof}

\begin{corollary}\label{maximal-fans-in-product}
If $\Sigma$ is maximal, then it is a product of a maximal fan in $N$ and a maximal fan in $N'$.
\end{corollary}

\subsection{Products of GIT quotients}
To show that the property of being projective plays well with taking products of quotients, we collect here a few facts about semistable points of products.

Suppose that $G$ acts on $\Aff^m$ and $G'$ acts on $\Aff^n$, and let $\chi$ be a character of $G$, $\chi'$ a character of $G'$.  Then $G\times G'$ acts naturally on $\Aff^{m+n}$, and the characters $\chi$ and $\chi'$ give characters of $G\times G'$ as well via pulling back along the projections onto $G,\;G'$.  Let $\chi''$ denote the product of these induced characters.  

We claim that in the case where $G$ and $G'$ are both tori, $(\Aff^{m+n})^{ss}_{\chi''} = (\Aff^m)^{ss}_\chi \times (\Aff^n)^{ss}_{\chi'}$.  To see this, observe that invariant sections of the $G$-linearized trivial bundle given by $\chi$ are exactly polynomial functions $f$ on $\Aff^m$ satisfying $gf = \chi(g)f$ for all $g \in G$; such functions are called \textit{$(G,\chi)$-invariant}.  

Suppose $(a,b) \in (\Aff^m)^{ss}_\chi \times (\Aff^n)^{ss}_{\chi'}$.  Then there exists a $(G,\chi)$-invariant polynomial $f(x)$ and a $(G',\chi')$-invariant polynomial $f'(y)$ such that $f(a) \ne 0$ and $f'(b) \ne 0$.  Then $ff'$ is $(G\times G', \chi'')$-invariant and $ff'(a,b) \ne 0$, so $(\Aff^{m+n})^{ss}_{\chi''} \subseteq (\Aff^m)^{ss}_\chi \times (\Aff^n)^{ss}_{\chi'}$.

Conversely, suppose $(a,b) \in (\Aff^{m+n})^{ss}_{\chi''}$.  Then there exists a $(G\times G', \chi'')$-invariant polynomial $f''(x,y)$ such that $f''(a,b) \ne 0$.  Write
\begin{gather*}
    f''(x,y)
    = \sum_{i,j}c_{i,j}x^iy^j
    = \sum_j\qty(\sum_i c_{i,j}x^i)y^j.
\end{gather*}
Then, for any $g \in G$, we have that
\begin{gather*}
    \sum_j\qty[g\qty(\sum_i c_{i,j}x^i)]y^j
    = (g,1)f''(x,y)
    = \chi''(g,1)f''(x,y)
    = \sum_j\qty[\chi(g)\qty(\sum_i c_{i,j}x^i)]y^j,
\end{gather*}
from which we have that $g\qty(\sum_i c_{i,j}x^i) = \chi(g)\qty(\sum_i c_{i,j}x^i)$ for all $j$.  But because $G$ is a torus, this implies that $gc_{i,j}x^i = \chi(g)c_{i,j}x^i$ for all $i,j$, i.e., $gx^i = \chi(g)x^i$ for all the $i,j$ for which $c_{i,j} \ne 0$.  By symmetry, we have $g'y^j = \chi'(g')y^j$ for all the $i,j$ for which $c_{i,j} \ne 0$.  Since $f''(a,b) \ne 0$, we must have $c_{i,j}a^ib^j \ne 0$ for some $i,j$.  This implies that $c_{i,j} \ne 0$, so that $x^i$ and $y^j$ are $(G,\chi)$- and $(G', \chi')$-invariant respectively.  We also have that $a^i \ne 0$ and $b^j \ne 0$, so by definition $(a,b) \in (\Aff^m)^{ss}_\chi \times (\Aff^n)^{ss}_{\chi'}$, and the claim is proved.

Since, for any $G$-stable $U \subseteq \Aff^m$ and $G'$-stable $V \subseteq \Aff^n$, we have that $(U \times V)/(G \times G') \cong U/G\,\times\,V/G'$, this completes the proof of
\begin{prop}
If $U/G$ and $V/G'$ are both GIT quotients, then $(U \times V)/(G \times G') \cong U/G\,\times\,V/G'$ is a GIT quotient as well.
\end{prop}

\subsection{Extending to higher dimensions}

\begin{proof}[Proof of Theorem \ref{all-maximal-fans-complete}]
Fix any $n \ge 3$.  Let $N_1 = \Z^3$, $N_2 = \Z^{n-3}$, let $R_1$ be the set of rays in $(N_1)_\R$ given by Proposition \ref{dim3-rays}, and let $R_2$ be the collection of rays in $(N_2)_\R$ generated by $e_1, -e_1, \dots, e_{n-3}, -e_{n-3}$, where $e_1, \dots, e_{n-3}$ is a basis of $N_2$.  We claim that the collection $R$ of rays in $(N_1 \oplus N_2)_\R \cong \R^n$ given by the images of $R_1$ and $R_2$ satisfies the conclusions of the theorem.

Suppose that $\Sigma$ is a maximal fan with respect to $R$.  Then, by Corollary \ref{maximal-fans-in-product} and induction, $\Sigma = \Sigma_1 \times \Sigma_{2,1} \times \cdots \times \Sigma_{2,n-3}$, where $\Sigma_1$ is a maximal fan with respect to $R_1$ and each $\Sigma_{2,i}$ is maximal with respect to $\{\Cone -e_i, \Cone e_i\}$.  Therefore, the corresponding toric variety $X_\Sigma$ is isomorphic to $X_{\Sigma_1} \times (\Pro^1)^{n-3}$.  From this it follows that $X_\Sigma$ is complete because every possible $X_{\Sigma_1}$ is complete and $\Pro^1$ is as well.  Since products of projective varieties are projective, whenever $X_{\Sigma_1}$ is projective, so is $X_\Sigma$.  The converse holds as well since $X_{\Sigma_1} \cong X_{\Sigma_1} \times (*)^{n-3}$ (where $*$ is any closed point of $\Pro^1$) is a closed subset of $X_{\Sigma_1} \times (\Pro^1)^{n-3} = X_\Sigma$, and hence is projective if $X_\Sigma$ is.  Therefore, both projective and non-projective $X_\Sigma$ occur.  Finally, since a product cone $\sigma \times \sigma'$ is simplicial if and only if both $\sigma$ and $\sigma'$ are simplicial, it follows that, in both the projective and non-projective cases, both simplicial and non-simplicial fans occur.
\end{proof}

\begin{proof}[Proof of Theorem \ref{all-maximal-nonprojective-fans-noncomplete}]
The proof proceeds in the same way, except that $N_1 \cong \Z^4$, $N_2 \cong \Z^{n-4}$, and $R_1$ is as in Proposition \ref{dim4-rays}.  Using the same notation as the previous proof, $X_\Sigma$ is complete (respectively projective) if and only if $X_{\Sigma_1}$ is, so we get that $X_\Sigma$ is either projective or non-complete, with both possibilities occurring.  Just as in the previous proof, in both cases, both simplicial and non-simplicial fans occur.
\end{proof}

\section{Examples}







\subsection{A maximal non-complete fan}


Consider the following points in $\R^4$:

\begin{align*}
  u_1 &= (-2, -1, 0, 1) \\
  u_2 &= (2, -1, 0, 1) \\
  u_3 &= (0, -2, -1, 1) \\
  u_4 &= (0, 2, -1, 1) \\
  u_5 &= (-1, 0, -2, 1) \\
  u_6 &= (-1, 0, 2, 1) \\
  u_7 &= (0, 0, 0, -1)
\end{align*}

Let $\rho_i$ be the ray generated by $u_i$, and consider the set $S$ of fans
whose rays are a subset of $\{\rho_i\}_{i = 1}^7$.  Then the fan $\Sigma$ with
maximal cones $(1237)$, $(1267)$, $(3427)$, $(3457)$, $(5617)$, $(5647)$,
$(1357)$, and $(2467)$ (where we have listed only the indices $i$ corresponding
to rays $\rho_i$ included in each cone) is non-complete, and yet is maximal in
$S$.

To see that $\Sigma$ is non-complete, it suffices to verify that $(0, 0, 0, 1)$
is not contained in any maximal cone of $\Sigma$.  To show that $\Sigma$ is
maximal in $S$, consider the following points in $\R^3$:

\begin{align*}
  v_1 &= (-2, -1, 0) \\
  v_2 &= (2, -1, 0) \\
  v_3 &= (0, -2, -1) \\
  v_4 &= (0, 2, -1) \\
  v_5 &= (-1, 0, -2) \\
  v_6 &= (-1, 0, 2)
\end{align*}

Let $V$ be the region of $\R^3$ enclosed inside the surface triangulated with
triangles $(123)$, $(126)$, $(342)$, $(345)$, $(561)$, and $(564)$ (where we
have listed only the indices $i$ corresponding to the vertices $v_i$ of each
triangle).

Then the region in $\R^4$ not in the support of $\Sigma$ is the cone over $V$,
where we identify $\R^3$ with the hyperplane $w = 1$ in $\R^4$ with coordinates
$(x,y,z,w)$.  The question of whether a cone can be added to $\Sigma$ while
remaining within $S$ and maintaining the fan condition is then equivalent to
whether there exists a convex polytope with vertices among $\{v_i\}_{i = 1}^6$
whose relative interior is completely contained within $V$.  However, the points
$v_i$ consist of three distinguished pairs, $\{v_1, v_4\}$, $\{v_2, v_5\}$, and
$\{v_3, v_6\}$, with the property that the relative interior of the line segment
across each pair lies entirely outside $V$.  Any full-dimensional convex
polytope $P$ with vertices among $\{v_i\}_{i = 1}^6$ must include at least four
of the $v_i$, and therefore must contain both vertices belonging to one of these
distinguished pairs.  Thus, the relative interior of $P$ escapes $V$.  This
shows that no member of $S$ containing $\Sigma$ can be complete, and it can also
be verified that all lower-dimensional cones compatible with $\Sigma$ are
already included in $\Sigma$.

It is of particular note that, by raising $V$ higher in the $w$ direction, the corresponding fan can be made to have the complement of its support arbitrarily small while still remaining maximal.

\section{Background on Convexity}

Let $N$ be a free abelian group of rank $n$ and let $M$ be the dual group $\Hom{\Z}{N}{\Z}$.  We denote the pairing between $n \in N$ and $m \in M$ by $\sprod{n}{m}$.  We define $N_\R \definedby N \otimes_\Z \R$, and similarly for $M_\R$ and write the natural extension of the duality pairing without modification.  We will commonly assume some fixed basis of $N$ has been chosen and equip $M$ with the dual basis to identify $N$ and $M$ with $\Z^n$ and $N_\R$ and $M_\R$ with $\R^n$.

If $x, y \in \R^n$, we denote by $[x,y]$ the closed line segment between $x$ and $y$.  Recall that a subset $K \subseteq \R^n$ is called \textit{convex} if, for all $x, y \in K$, $[x,y] \subseteq K$, and that $K$ is called a \textit{cone} if $\lambda x \in K$ for all $x \in K$ and $\lambda \in \R_{\ge 0}$.  By $\Conv X$ we denote the smallest convex set containing $X$, and by $\Cone X$ we denote the smallest \textit{convex} cone containing $X$.

A convex set $K \subseteq \R^n$ is a \textit{convex polytope} if $K = \Conv \{x_1, \dots, x_s\}$ for some $x_1, \dots, x_s \in \R^n$, and is, moreover, called a \textit{lattice polytope} if one can take the $x_i$ to be in $\Z^n$ or a \textit{rational polytope} if one can take the $x_i$ to be in $\Q^n$.  Similarly, a convex cone $\sigma \subseteq \R^n$ is called \textit{polyhedral} if $\sigma = \Cone \{x_1, \dots, x_s\}$ for some $x_1, \dots, x_s \in \R^n$, and is called \textit{rational polyhedral} if, moreover, the $x_i$ can be chosen from $\Z^n$.

An important subclass of convex sets are the \textit{affine} sets, that is, those subsets $A \subseteq \R^n$ such that, whenever $x \neq y \in A$, the entire line through $x$ and $y$ is contained in $A$.  We let $\Affine X$ denote the smallest affine set containing $X$.

If $K$ is convex, then it is naturally a subset of $\Affine K$, and we define the \textit{relative interior}, \textit{relative closure}, and \textit{relative boundary} of $K$ to be the topological interior, closure, and boundary of $K$ as a subset of $\Affine K$.  In addition, we define the dimension $d$ of $K$ to be the dimension of $\Affine K \cong \R^d$.

If $K$ is a closed convex set, then a hyperplane $H$ is said to \textit{support} $K$ if $K$ is contained in one of the closed halfspaces of $H$ and $K \cap H \ne \emptyset$.

If $K$ is any convex set, then a \textit{face} of $K$ is a subset $F \subseteq K$ with the property that, whenever $a,b \in K$ are such that $F \cap \relint [a,b] \ne \emptyset$, then in fact $a, b \in F$.  $F$ is an \textit{exposed face} of $K$ if $F = K \cap H$ for some supporting hyperplane $H$ of $K$.  It is easy to see that every exposed face of $K$ is a face of $K$, and, moreover, if $K$ is a convex polytope or a polyhedral cone, then the converse holds as well.  We will write $F \le K$ to mean that $F$ is a face of $K$.

We say a convex cone $\sigma$ is \textit{strictly convex} if $\{0\}$ is a face of $\sigma$.

If $v \in \R^n$ and $a \in \R$, we define the hyperplane $H_{v,a}$ and the
half-spaces $H_{v,a}^+$ and $H_{v,a}^-$ by
\begin{align*}
  H_{v,a} &= \{x \in \R^n: \sprod{x}{v} = a\} \\
  H_{v,a}^+ &= \{x \in \R^n: \sprod{x}{v} \ge a\} \\
  H_{v,a}^- &= \{x \in \R^n: \sprod{x}{v} \le a\}.
\end{align*}
We abbreviate $H_{v,0}$ by $H_v$ and similarly for $H^+$ and $H^-$.

Given a closed convex set $K \subseteq \R^n$, its \textit{support function} $h_K: \R^n
\to \R$ is defined by
\begin{gather*}
  h_K(v) = \mathrm{sup}_{x \in K}\sprod{x}{v},
\end{gather*}
or, equivalently,
\begin{gather*}
  h_K(v) = \mathrm{inf}\{a \in \R: K \subseteq H_{v,a}^-\}.
\end{gather*}
Then $h_K$ is a convex function \cite[Section 1.7]{Schneider}, and hence is
continuous \cite[Theorem 1.5.1]{Schneider}.

\section{Proof of Theorem \ref{rank2-tori-projective}}

Since the support function of a convex set is continuous, we have the following:

\begin{lemma}\label{zero-in-interior}
  Let $X \subseteq \R^n$ be such that $\mathrm{cone}(X) = \R^n$.  Then the
  origin is an interior point of $\mathrm{conv}(X)$.
\end{lemma}
\begin{proof}
  Let $K = \mathrm{conv}(X)$.  Since the interiors of $K$ and $\ol{K}$ coincide
  \cite[Theorem 1.1.14 (a)]{Schneider}, we may assume without loss of generality that
  $K$ is closed.  Let $r = \mathrm{min}\{h_K(u): u \in S^{n-1}\}$ (the minimum
  is attained because $h_K$ is continuous and $S^{n-1}$ is compact) and let
  $u_0$ be such that $r = h_K(u_0)$.  We claim $r > 0$.  If not, then $K
  \subseteq H^-_{u_0,r} \subseteq H_{u_0}^-$ and, since $H_{u_0}^-$ is a convex
  cone, we have that $\mathrm{cone}(K) = \mathrm{cone}(X) \subseteq H_{u_0}^-$,
  contrary to our assumptions on $X$.  Finally, since $K = \cap_{u \in
    S^{n-1}}H_{u, h_K(u)}^-$ \cite[Corollary 1.3.5]{Schneider}, we have that $B_r(0)
  \subseteq K$, so $0 \in \mathrm{int}K$.
\end{proof}

\begin{lemma}\label{rational-separating-plane}
  Let $\sigma$ be a strictly-convex polyhedral cone in $\R^n$.  Then there
  exists $v \in \Q^n$ such that $\sigma \subseteq H_v^+$ and $\sigma \cap H_v =
  \{0\}$.
\end{lemma}
\begin{proof}
  Since $\sigma$ is polyhedral, $\sigma = \mathrm{cone}(x_1, \dots, x_s)$ for
  some $x_1, \dots, x_s \in \R^n$.  Let
  \begin{gather*}
    U = \{v \in \R^n: \sigma \subseteq H_v^+
    \text{ and } \sigma \cap H_v = \{0\}\}.
  \end{gather*}
  It is easy to check that also
  \begin{gather*}
    U = \{v \in \R^n: \sprod{v_i}{v} > 0 \;\forall i = 1, \dots, s\},
  \end{gather*}
  and thus $U$ is open.  Since $\sigma$ is strictly convex, $U$ is non-empty, so
  it contains a rational point.
\end{proof}

\begin{lemma}\label{rational-slice}
  Let $\sigma$ be a strictly-convex rational polyhedral cone in $\R^n$.  Then
  there exists $v \in \Q^n$ such that $\sigma \cap H_{v,1}$ is a lattice
  polytope.
\end{lemma}
\begin{proof}
  Since $\sigma$ is a rational polyhedral cone, $\sigma = \mathrm{cone}(x_1,
  \dots, x_s)$ for some $x_1, \dots, x_s \in \Q^n$.  By lemma
  \ref{rational-separating-plane}, there exists $v \in \Q^n$ such that
  $\sprod{x_i}{v} > 0$ for all $i = 1, \dots, s$.  Thus, $H_{v,1} \cap
  \mathrm{cone}(x_i) = \{y_i\}$, where $y_i = x_i/\sprod{x_i}{v}$.  It follows
  that $\sigma \cap H_{v,1} = \mathrm{conv}(y_1, \dots, y_s)$, and therefore
  that the vertices of the convex polytope $K \definedby \sigma \cap H_{v,1}$
  are among $y_1, \dots, y_s$, which all have rational coordinates.  We can thus
  rescale $v$ to clear denominators so that $K$ has integral vertices.
\end{proof}
\begin{theorem}\label{fans-with-big-cones-are-projective}
  Let $\Sigma$ be a fan in $\R^n$ with $t$ rays such that contains a cone
  $\sigma$ with $t-1$ rays and such that $\Sigma$ is maximal among fans with the
  same rays.  Then $\Sigma$ is the normal fan of a lattice polytope, hence
  $X_\Sigma$ is projective.
\end{theorem}
\begin{proof}
  By Lemma \ref{rational-slice}, there is some $v \in \Q^n$ such that $K
  \definedby \sigma \cap H_{v,1}$ is a lattice polytope.  Let $\rho$ be the sole
  ray of $\Sigma$ not in $\sigma(1)$.  Then $\rho \cap H_v^+ = 0$, since
  otherwise the cone on the rays of $\Sigma$ could not be all of $\R^n$ - our standing assumption on the fans we consider; hence $\rho \cap H_{v,1} =
  \emptyset$.  Let $u_\rho$ be the generator of $\rho$ and let $P =
  \mathrm{conv}(K \cup \{u_\rho\})$.  Then $\mathrm{cone}(P) = \R^n$ since
  $\Sigma$ is complete, hence $0 \in \mathrm{int}P$.

  We will show that $\Sigma$ is the fan over the proper faces of $P$.  Given this,
  by \cite[Lemma 2.2.3]{Schneider}, each cone of $\Sigma$ is the normal cone of a face
  of the dual polytope $P^*$, making $\Sigma$ the normal fan of $P^*$.  While
  this need not be a lattice polytope, it must have rational vertices since $P$
  does; hence, a suitable multiple $kP^*$ will be a lattice polytope, and
  $\alpha P^*$ has the same normal fan as $P^*$ for all $\alpha > 0$.

  It is well-known that the faces of $P$ have two types: the faces of $K$ and
  the faces $\mathrm{conv}(F \cup \{u_\rho\})$, where $F$ is a face of $K$.
  Clearly $\sigma = \mathrm{cone}(K)$ is the cone over a proper face of $P$, so
  let $\tau$ be any other cone of $\Sigma$ besides $\sigma$.  Then $\tau$ either
  contains $\rho$ or not.  In either case, $\tau \cap \sigma$ must be a face
  $\nu$ of $\sigma$ by the fan condition, so $F = \nu \cap H_{v,1}$ must be a
  face of $K$.  Thus, either $\tau = \mathrm{cone}(F)$ or $\tau =
  \mathrm{cone}(F \cup \{u_\rho\})$, and $\tau$ is the cone over a proper face
  of $P$.

  On the other hand, by \cite[Lemma 1.1.8]{Schneider}, it follows that each ray based at 0
  intersects the boundary of $P$ in exactly one point, so that there is a
  bijection between boundary points of $P$ and rays from 0.  Since the proper
  faces of $P$ satisfy conditions analogous to the fan conditions, the cones
  over the proper faces of $P$ form a fan which contains $\Sigma$ as was just
  shown.  By maximality, these fans are equal.
\end{proof}

\begin{corollary}
  The conclusion of Theorem \ref{fans-with-big-cones-are-projective} with the maximality condition replaced by the condition that
  $\Sigma$ is
  complete.
\end{corollary}

\begin{corollary}\label{quotients-by-2-tori-are-projective}
  Any complete toric variety $X$ of dimension $d$ whose fan has $d+2$ rays is
  projective.
\end{corollary}
\begin{proof}
  If $X$ is simplicial, then this has already been shown by \cite{acampo2003orbit}.  Otherwise, the fan of $X$ has at least one cone with at least $d+1$
  rays.  Since $X$ would be affine, and hence non-complete, if all rays were in
  the same cone, one cone has exactly $d+1$ rays and the previous theorem
  applies.
\end{proof}

\section{Proof of Theorem \ref{3d-fans-complete}}

\begin{theorem}\label{triangulate-simple-closed-curve-interior}
Let $C \subseteq \R^2$ be a simple closed curve consisting of a finite number of line segments.  Then the region enclosed by $C$ can be triangulated in such a way that all the triangle vertices are among the nodes of $C$.
\end{theorem}
\begin{proof}
  For simplicity, orient $C$ positively with respect to the region $D$ it
  encloses.  Since the sum of the exterior angles of this curve must be $2\pi$,
  at least one node $v_2$ must have positive external angle.  Let $v_1, v_3$ be
  the nodes adjacient to $v_2$.  Then either the interior of the line segment
  $[v_1, v_3]$ is entirely contained in $D$ or it is not.  If it is, divide off
  the triangle $\mathrm{conv}(v_1, v_2, v_3)$.  The remaining portion of the
  region is of the same type as $D$, but its boundary contains one fewer line
  segment.  Since the theorem is clear when $C$ consists of three line segments
  (and $C$ cannot be simple with fewer line segments), the theorem is true by
  induction.  So suppose instead that $[v_1, v_3]$ intersects $C$ at some other
  point $p$.  
  \begin{claim} \label{claim.hard}
There is a node $w$ on $C$ such that the interior of
  $[v_2, w]$ lies entirely within $D$. 
  \end{claim}
  The proof of this claim is quite subtle and we defer its proof to the end of this section.
Given the claim  we can divide $D$ into two
  regions along $[v_2, w]$ and again apply induction to both pieces, since they
  necessarily have fewer line segments bounding them than $D$ does.  
 \end{proof}

\begin{prop}\label{triangulate-bounded-region}
  If an open region $D \subseteq \R^2$ is bounded and $\partial D$ consists of a
  finite number of line segments, then $D$ can be triangulated without adding
  any more nodes to the boundary.
\end{prop}
\begin{proof}
  We may assume that $D$ is connected since otherwise we can work on one
  connected component at a time.  Although the boundary $C \definedby \partial
  D$ may not be a simple closed curve as show in Figure \ref{fig.nonclosed},
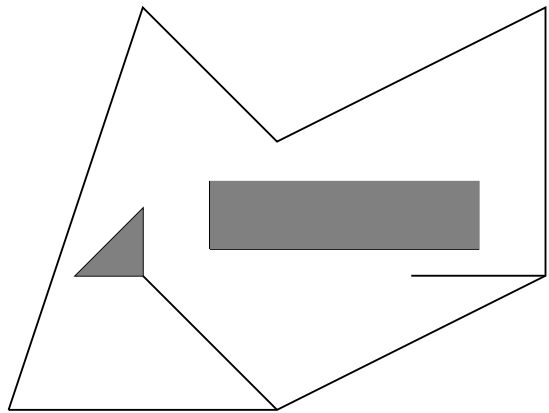
\begin{figure}[h!]\centering
\resizebox{3in}{!}{
\begin{tikzpicture}
\draw (0,0) -- (2,0) -- (4,1) -- (4,3) -- (2,2) -- (1,3) -- (0,0);
\draw (2,0) -- (1,1);
\draw (1,1) -- (1,1.5) -- (0.5,1) -- (1,1);
\fill[gray] (1,1) -- (1,1.5) -- (0.5,1) -- (1,1);
\draw (4,1) -- (3,1);
\draw (1.5,1.2) -- (3.5,1.2) -- (3.5,1.7) -- (1.5,1.7) -- (1.5,1.2);
\fill[gray] (1.5,1.2) -- (3.5,1.2) -- (3.5,1.7) -- (1.5,1.7) -- (1.5,1.2);
\end{tikzpicture}
}
\caption{Example of a boundary curve which is not simple.} \label{fig.nonclosed}
\end{figure}
  we will modify the proof of Theorem
  \ref{triangulate-simple-closed-curve-interior} in the following way.  When
  selecting $v_2$ we will double the line segments of $C\setminus\partial\ol{D}$
  and consider only those components of $C$ adjacent to the unbounded component
  of $\R^2\setminus D$ as indicated in the Figure \ref{fig.doubled}.
\begin{figure}[h!]\centering
\resizebox{3in}{!}{
\begin{tikzpicture}
\draw (0,0) -- (2,0) -- (1,1) -- (0.5,1) -- (1.05,1.5) -- (1.05,1.05) -- (2.1,0) -- (4,1) -- (3,1) -- (3,1.05) -- (4, 1.05) -- (4,3) -- (2,2) -- (1,3) -- (0,0);
\draw (1.5,1.2) -- (3.5,1.2) -- (3.5,1.7) -- (1.5,1.7) -- (1.5,1.2);
\end{tikzpicture}
}
\caption{Doubling the line segments in $C \setminus \partial \overline{D}$.} \label{fig.doubled} 
\end{figure}
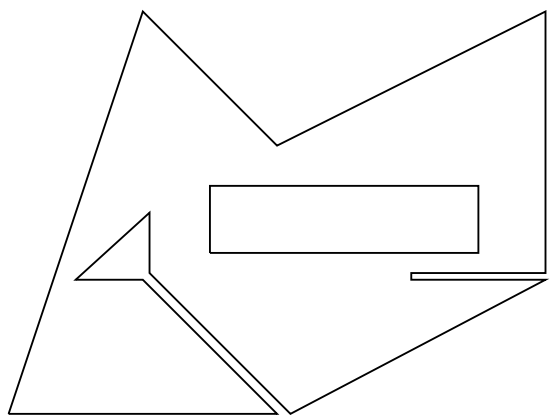

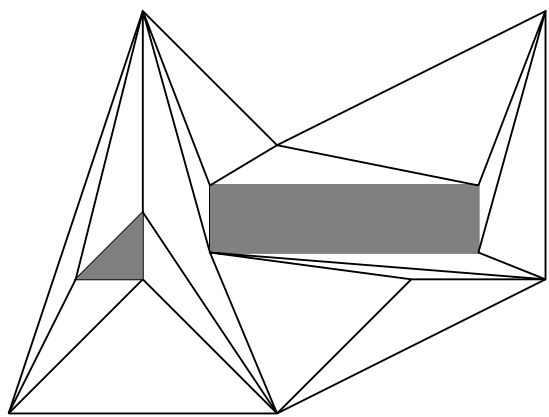
\begin{figure}[h!]\centering
\resizebox{3in}{!}{
\begin{tikzpicture}
\draw (0,0) -- (2,0) -- (4,1) -- (4,3) -- (2,2) -- (1,3) -- (0,0);
\draw (2,0) -- (1,1);
\draw (1,1) -- (1,1.5) -- (0.5,1) -- (1,1);
\fill[gray] (1,1) -- (1,1.5) -- (0.5,1) -- (1,1);
\draw (4,1) -- (3,1);
\draw (1.5,1.2) -- (3.5,1.2) -- (3.5,1.7) -- (1.5,1.7) -- (1.5,1.2);
\fill[gray] (1.5,1.2) -- (3.5,1.2) -- (3.5,1.7) -- (1.5,1.7) -- (1.5,1.2);
\draw (0,0) -- (1,1) (1,1.5) -- (2,0) (1,3) -- (1.5,1.7) (3,1) -- (2,0) (4,3) -- (3.5,1.2) (4,3) -- (3.5,1.7) (3.5,1.7) -- (2,2) (2,2) -- (1.5,1.7) (1,3) -- (0.5,1) (0.5,1) -- (0,0) (4,1) -- (1.5,1.2) (1,1.5) -- (1,3) (1.5,1.2) -- (1,3) (2,0) -- (1.5,1.2) (1.5,1.2) -- (3,1) (4,1) -- (3.5,1.2);
\end{tikzpicture}
}
\caption{Illustration of a triangulation of the region $D$ obtained using the method of proof of Proposition \ref{triangulate-bounded-region}.} \label{fig.triangulated}
\end{figure}

For the remainder of the proof, we
  will take the entire boundary $C$ into account.  Observe that Lemma
  \ref{nonsmooth-only-at-node} doesn't rely on the full strength of $C$ being a
  simple closed curve, only that each segment is disjoint from the relative
  interior of the others, which still holds in this modified case.  The only way
  that the result can fail to be a triangulation of $D$ is if $C$ remains
  disconnected at the end of the process (note that in the last case considered
  by the theorem, a new segment can be added connecting different components of
  $C$).  In this case, each component of $C$ not adjoining the unbounded
  component of $\R^2\setminus D$ must be contained in a triangle.  By letting
  $v_1, v_2, v_3$ be the vertices of this triangle, the last part of the proof
  of Theorem \ref{triangulate-simple-closed-curve-interior} shows that we can
  add a line segment connecting $v_2$ to one of the other components of $C$.
  Then apply the theorem again to the portion of $D$ inside this triangle.
  Since $C$ can have only finitely many components, this process will eventually
  terminate, showing that $D$ can be triangulated as illustrated in Figure
  \ref{fig.triangulated}.
\end{proof}
\begin{theorem}
  Let $\Sigma$ be a fan in $\R^3$ maximal with respect to having the same rays
  as $\Sigma$.  Then either $\ol{\R^3\setminus\mathrm{supp}\,\Sigma}$ is not
  contained in any open half-space or $\Sigma$ is complete.
\end{theorem}
\begin{proof}
  Suppose $\ol{\R^3\setminus\mathrm{supp}\,\Sigma}$ is contained in the
  half-space $\mathrm{int}\,H^+_v$.  Then $(\R^3\setminus\mathrm{supp}\,\Sigma)
  \cap H_{v,1}$ is a bounded open region in $H_{v,1} \cong \R^2$ whose boundary
  consists of a finite number of line segments.  Applying the corollary shows
  that this region can be triangulated without adding any nodes.  By taking the
  cones over these triangles, we have shown $\R^3\setminus\mathrm{supp}\,\Sigma$
  can be triangulated without adding any more rays.  But then the maximality of
  $\Sigma$ shows that $\Sigma$ must already have been complete.
\end{proof}
\subsection{Proof of Claim \ref{claim.hard}}
  If we knew that $C$ enclosed a convex region the claim would be immediate because any ray from $v_1$ to any another node $w$ of $C$ could only hit $C$ in
  a single point. The dificulty is that in our case such a ray could hit $C$ in other points, not necessarily nodes. 
  
  \begin{lemma}\label{distance-to-line-behavior}
    If $L = \{(x,y): (a,b)\cdot(x,y) = c\}$ is a line not passing through the
    origin and $L_\lambda$ is the line through the origin with slope $\lambda$,
    then
    \begin{gather*}
      f_L(\lambda)
      \definedby\mathrm{dist}^2(0, L \cap L_\lambda)
      = \frac{c^2(1+\lambda^2)}{(a + b\lambda)^2}.
    \end{gather*}
    Moreover, if $M$ is another line such that all the derivatives of $f_L$ and
    $f_M$ agree at some point $\lambda_0$, then $L = \pm M$.
  \end{lemma}
  \begin{proof}
    The first assertion is easily verified.  For the second, since $f_L$ and
    $f_M$ are both rational functions of $\lambda$, they are both analytic on
    their domains, so if all their derivatives agree at $\lambda_0$, then $f_L =
    f_M$ on some neighborhood of $\lambda_0$.  But since these functions are
    rational, this implies they are everywhere equal.

    Writing $M = \{(x,y): (a',b')\cdot(x,y) = c'\}$, we see that
    \begin{gather*}
      \frac{c^2(1+\lambda^2)}{(a + b\lambda)^2}
      = \frac{c'^2(1+\lambda^2)}{(a' + b'\lambda)^2}
    \end{gather*}
    is equivalent to the system
    \begin{align*}
      c^2a'^2 &= c'^2a^2 \\
      c^2a'b' &= c'^2ab \\
      c^2b' &= c'^2b.
    \end{align*}
    By the assumption that $L$ does not pass through the origin, we have that $c \ne 0$.  Letting $\mu = c'/c$, we see that the solutions of this system are
    \begin{gather*}
      a' = \mu a,\; b' = \mu b
      \qqtext{and}
      a' = -\mu a,\; b' = -\mu b,
    \end{gather*}
    and these are equivalent to $M = L$ and $M = -L$ respectively.
  \end{proof}

  \begin{lemma}\label{nonsmooth-only-at-node}
    With the notation of the previous lemma, let $L$ and $M$ be lines not
    passing through the origin and let $a,b,c,d \in \R$ be such that $a \le b
    \le c \le d$, $a < c$, $b < d$, and such that the line segments $S_L
    \definedby \{L \cap L_\lambda: a \le \lambda \le c\}$ and $S_M \definedby
    \{M \cap L_\lambda: b \le \lambda \le d\}$ are closed, lie in the open
    half-plane $x > 0$, and that each is disjoint from the relative interior of the
    other.  Consider the function
    \begin{gather*}
      g(\lambda) = \begin{cases}
        f_L(\lambda) & a \le \lambda < b \\
        \min\{f_L(\lambda), f_M(\lambda)\} & b \le \lambda \le c \\
        f_M(\lambda) & c < \lambda \le d.
      \end{cases}
    \end{gather*}
    Then we claim that 
    \begin{enumerate}
    \item
      If $g$ is smooth, then $L = M$ and $b = c$.

    \item
      The only points where $g$ can fail to be smooth are $b$ and $c$.  If $g$
      is not smooth at $b$, then the point of $L_b \cap (L \cup M)$ closest to
      the origin is an endpoint of $S_M$, and, similarly, if $g$ is not smooth
      at $c$, then the point of $L_c \cap (L \cup M)$ closest to the origin is
      an endpoint of $S_L$.
    \end{enumerate}
  \end{lemma}
  \begin{proof}
    From the definition it is clear that $g$ is smooth on $[a,b)$ and $(c,d]$
    and continuous on $(b,c)$.

    Suppose that $g$ is smooth.  Then, since $g$ is continuous, there must be
    some point $\lambda_0 \in [b,c]$ such that $f_L(\lambda_0) =
    f_M(\lambda_0)$; otherwise, $g$ would be discontinuous at either $b$ or $c$.
    Since $f_L$ and $f_M$ are rational, they can only be equal to each other at
    a finite number of points, so there is some neighborhood of $\lambda_0$ such
    that $g = f_L$ on $(\lambda_0-\varepsilon, \lambda_0]$ and $g = f_M$ on
      $[\lambda_0, \lambda_0+\varepsilon)$ or vice-versa.  In either case, $g$
        being smooth at $\lambda_0$ must mean that the derivatives of $f_L$ and
        $f_M$ agree at $\lambda_0$, so by the previous lemma, $L = \pm M$.  By
        our assumption that $S_L$ and $S_M$ both lie in the same half-plane, we
        must have $L = M$, and since we assumed that the relative interiors of
        $S_L$ and $S_M$ are disjoint, it must be that $b = c$.

    We now show that $g$ must be smooth on $(b,c)$.  Since $f_L$ and
    $f_M$ are both smooth, a point where $g$ is not smooth can only occur at
    some $\lambda_0 \in (b,c)$ where $f_L = f_M$.  But then the interior of
    $S_L$ intersects the interior of $S_M$, contrary to our assumptions.

    Finally, we show that if $g$ is not smooth at $b$, then the point of $L_b
    \cap (L \cup M)$ closest to the origin is an endpoint of $S_M$ (the other
    case being essentially the same).  Clearly, if $f_M(b) < f_L(b)$, then this
    is the case by the definitions of $f_L$ and $f_M$.  It is impossible to have
    $f_L(b) < f_M(b)$, since otherwise, $f_L$ and $f_M$ being continuous, we
    would have $g = f_L$ on a neighborhood of $b$, and $g$ would be smooth at
    $b$.  So consider the case when $f_L(b) = f_M(b)$.  We claim that in this
    case $b = c$, so that the point of $L_b \cap (L \cup M)$ closest to the
    origin is an endpoint of both $S_L$ and $S_M$.  This is so since $f_L(b) =
    f_M(b)$ is equivalent to $S_L$ and $S_M$ and $L_b$ intersecting at a common
    point, and we assumed that $S_M$ does not intersect the interior of $S_L$.
  \end{proof}

  We may make a rigid transformation to assume that $v_2 = 0$ is the origin and
  that $v_1$ and $v_3$ are in the open  half-plane $x > 0$.  Let $\lambda_1$ and
  $\lambda_3$ be the slopes of the lines through the origin and $v_1, v_3$
  respectively.  For any $\lambda \in (\lambda_1, \lambda_3)$, let
  $\rho_\lambda$ be the ray through the origin with slope $\lambda$ in the right
  half-plane, and let $f(\lambda) \definedby \mathrm{dist}^2(0, C \cap
  \rho_\lambda)$.  Then $f$ is a minimum of functions of the form
  $f_M|_{[a,b]}$, one for each line segment of $C$.  Since there are only
  finitely many of these and since these are all rational functions, there are
  only a finite number of points where at least two of these functions are
  equal.  Therefore, we can divide $(\lambda_1, \lambda_3)$ into a finite
  number of intervals $J_i, i \in I$ such that 
  \begin{enumerate}
  \item
    the ray $\rho_\lambda$ intersects $C$ in a single segment $S_i$ of the
    boundary for all $\lambda \in J_i$,

  \item
    $f = f_{L_i}$ on $J_i$, where $L_i$ is the line containing $S_i$, i.e.,
    $S_i$ is closer to the origin than any other boundary segment along the rays
    $\rho_\lambda$, $\lambda \in J_i$, and

  \item
    the intervals $J_i$ are disjoint except possibly at their endpoints.
  \end{enumerate}

  If any of the intervals $J_i$ consists of only a single point, then this
  indicates that the line $L_i$ containing the boundary segment $S_i$ intersects
  the origin.  Then the endpoint $w$ of $S_i$ closer to the origin must be
  strictly closer than any other point on the boundary, otherwise the curve $C$
  would intersect itself.  Therefore, $w$ must be the node of $C$ desired.

  If none of the intervals $J_i$ consists of a single point, then we claim $f$
  cannot be smooth.  If it were, then by applying Lemma
  \ref{nonsmooth-only-at-node} at each adjacient pair of intervals $J_i$,
  $J_{i+1}$, we see that all the segments $S_i$ are contained in the same line
  $L$.  Since this line $L$ intersects $[v_1, v_3]$ at the point $p$, it must
  also intersect either $[v_1, v_2]$ or $[v_2, v_3]$, contradicting the
  assumption that $C$ was simple.

  Thus, let $\lambda_0$ be a point where $f$ is not smooth.  It must be the case
  that $\lambda_0$ lies at a boundary between some $J_i$ and $J_{i+1}$.  We
  claim that the point of $\rho_{\lambda_0} \cap C$ closest to the origin must
  be a node.  Otherwise, both segments $S_i$ and $S_{i+1}$ extend past the point
  at slope $\lambda_0$.  Thus, there exists some $c > \lambda_0$ such that $S_i
  \cap \rho_c \neq \emptyset$ and there exists some $b < \lambda$ such that
  $S_{i+1} \cap \rho_b \neq \emptyset$.  Also pick any $a \in J_i$ and $d \in
  J_{i+1}$.  Apply Lemma \ref{nonsmooth-only-at-node} (noting that the $g$ of
  the lemma is equal to $f$ in a neighborhood of $\lambda_0$) to get a
  contradiction, since the supposed point of nonsmoothness is not at $b$ or $d$.
  This proves the claim.

\section{Discussion of Methods of Computation}
To enumerate all the maximal open subsets,
it suffices to find all collections of distinguished polytopes satisfying (i)
and (ii) from Theorem \ref{BBS-good-quotient-characterization}, and then find those maximal with respect to the ``saturated in'' relation in
\ref{saturation}.  For even small values of $n$ this can be quite
computationally difficult: there are, in general, $2^n$ distinguished polytopes,
so $2^{2^n}$ collections to consider.  Our technique for finding those
collections satisfying (i) and (ii) is as follows:

Given an arbitrary collection $\Pi$ of distinguished polytopes, we can enlarge
$\Pi$ by repeatedly adding those polytopes required by (i) and (ii) to obtain
the smallest collection $\ol{\Pi}$ of polytopes satisfying (i) and (ii).  Then,
to enumerate all such collections, we first form $\ol{\{P\}}$ for each
distinguished polytope $P$, then $\ol{\ol{\{P\}} \cup \{Q\}}$ for each $Q \notin
\ol{\{P\}}$, and so on, building up the lattice of suitable collections from the
bottom up.  At each step, the computations of $\ol{C \cup Q}$ for the various $Q \notin C$ can be done in parallel, which we take advantage of.  To
further optimize performance, we perform all the necessary convexity
computations up front, storing the results for later use.

Instead of directly determining which of the corresponding good quotients is projective, we instead use toric GIT to separately enumerate the projective quotients, comparing the results by comparing the corresponding irrelevant ideals (for which we made use of Macaulay2).  For enumerating the GIT quotients, we use Cox, Little, and Schenck's
characterization of the secondary cone.  By determining the
minimal intersections between convex cones spanned by subsets of the dual
characters, we can determine the rays of the secondary fan.  Although we do not
find the higher-dimensional cones of the secondary fan explicitly, we can find
at least one character from each of these cones by taking positive linear
combinations of ray generators, where the number of rays used varies from two to
the dimension of the character lattice.  We then compute the associated polytope
for each character, which allows us to find the corresponding irrelevant ideal.

Once we have descriptions of all the maximal open subsets and their irrelevant ideals, we can use these to determine the corresponding fans.  It is straightforward to check if each fan is simplicial, but much less so if it is complete or not.  We utilize two approaches to check completeness: in one, we generate points at random in $\R^n$ according to a radially-symmetric Gaussian distribution and check to see whether each point is contained in a least one cone in the fan.  While this cannot be used to prove a fan is complete, in the case the fan is non-complete, it can usually be used to verify the non-completeness quickly.  The other approach is to compute the intersection of the complements of the cones.  Since the complement of each cone is a union of open half-spaces, this reduces to checking whether each of a (usually very large) collection of intersections of half spaces is empty.  We used this second approach to verify the completeness of all the maximal fans in Proposition \ref{dim3-rays}.  This approach, however, was computationally infeasible for the fans in Proposition \ref{dim4-rays}; fortunately, as we verified using the probabilistic approach, all the non-projective fans in that case are non-complete, and the remaining fans are complete since they are projective.  As a byproduct, we can estimate the fraction of the unit sphere in the support of each fan, and have tabulated these in the appendix.

Besides the use of Macaulay2 to test equality of the irrelevant ideals, all of our computations were implemented in Python + SAGE, relying on SAGE's implementations of convexity computations and linear algebra.  We expect much better performance could be achieved by re-implementing our algorithm in a faster language such as C or Rust.


\bibliographystyle{plain}
\bibliography{ref}

\pagebreak

\appendix
\section{Tables of open subsets with good quotients}
\subsection{Construction in Proposition \ref{dim3-rays}}

In the following table, the column headings are as follows: ``ND'' = non-degenerate, ``GIT'' = arising from GIT (which, as we have remarked, is equivalent to the corresponding quotient being projective), and ``S'' = simplicial.  Note that we are primarily interested in the non-degenerate quotients, so we have not computed the corresponding fans for the degenerate quotients and do not know whether these are simplicial or not.  All of the maximal fans in this construction were verified to be complete.

\begin{longtable}{l | c | c | c}
\hline
Subset & ND & GIT & S \\
\hline
$X \setminus (Z(b, d) \cup Z(a, f) \cup Z(a, c, d, e) \cup Z(b, f))$ & Yes & Yes & Yes \\
$X \setminus (Z(d) \cup Z(a, c, e, f) \cup Z(b, f))$ & Yes & Yes & Yes \\
$X \setminus (Z(a, e) \cup Z(a, f) \cup Z(b, c, d, f) \cup Z(c, e))$ & Yes & Yes & Yes \\
$X \setminus (Z(b, d) \cup Z(a, c, e) \cup Z(c, d) \cup Z(b, f))$ & Yes & Yes & Yes \\
$X \setminus (Z(b) \cup Z(f) \cup Z(a, c, d, e))$ & Yes & Yes & Yes \\
$X \setminus (Z(c, d) \cup Z(b, f) \cup Z(b, d) \cup Z(a, c, e) \cup Z(a, e, f))$ & Yes & Yes & Yes \\
$X \setminus (Z(b, d, f) \cup Z(a) \cup Z(c, e))$ & Yes & Yes & Yes \\
$X \setminus (Z(b, c, d) \cup Z(a, f) \cup Z(a, e) \cup Z(c, e))$ & Yes & Yes & Yes \\
$X \setminus (Z(a, b, e, f) \cup Z(c, e) \cup Z(d))$ & Yes & Yes & Yes \\
$X \setminus (Z(c) \cup Z(a, b, e, f) \cup Z(d))$ & Yes & Yes & Yes \\
$X \setminus (Z(a, f) \cup Z(a, e) \cup Z(c, d, e) \cup Z(b, f))$ & Yes & Yes & Yes \\
$X \setminus (Z(b, c, d) \cup Z(a, f) \cup Z(a, e) \cup Z(b, f))$ & Yes & Yes & Yes \\
$X \setminus (Z(b) \cup Z(c, d) \cup Z(a, c, e, f))$ & Yes & Yes & Yes \\
$X \setminus (Z(b) \cup Z(a, f) \cup Z(c, d, e))$ & Yes & Yes & Yes \\
$X \setminus (Z(c, e) \cup Z(b, d, f) \cup Z(b, c, d) \cup Z(a, f) \cup Z(a, e))$ & Yes & Yes & Yes \\
$X \setminus (Z(a) \cup Z(b, c, d, e) \cup Z(b, f))$ & Yes & Yes & Yes \\
$X \setminus (Z(c, d, e) \cup Z(b, f) \cup Z(b, c, d) \cup Z(a, f) \cup Z(a, e))$ & Yes & Yes & Yes \\
$X \setminus (Z(b) \cup Z(d) \cup Z(a, c, e, f))$ & Yes & Yes & Yes \\
$X \setminus (Z(a) \cup Z(c, d, e) \cup Z(b, f))$ & Yes & Yes & Yes \\
$X \setminus (Z(b, c, d) \cup Z(a, f) \cup Z(e))$ & Yes & Yes & Yes \\
$X \setminus (Z(b) \cup Z(a, f) \cup Z(a, c, d, e))$ & Yes & Yes & Yes \\
$X \setminus (Z(c) \cup Z(a, b, d, f) \cup Z(e))$ & Yes & Yes & Yes \\
$X \setminus (Z(e) \cup Z(a, b, d, f) \cup Z(c, d))$ & Yes & Yes & Yes \\
$X \setminus (Z(c) \cup Z(b, d) \cup Z(a, b, e, f))$ & Yes & Yes & Yes \\
$X \setminus (Z(b, d) \cup Z(a, f) \cup Z(c, d, e) \cup Z(b, f))$ & Yes & Yes & Yes \\
$X \setminus (Z(c, e) \cup Z(c, d) \cup Z(b, d, f) \cup Z(a, e) \cup Z(a, b, f))$ & Yes & Yes & Yes \\
$X \setminus (Z(a, f) \cup Z(a, e) \cup Z(b, c, d, e) \cup Z(b, f))$ & Yes & Yes & Yes \\
$X \setminus (Z(c, d, e) \cup Z(b, f) \cup Z(b, d) \cup Z(a, f) \cup Z(a, c, e))$ & Yes & Yes & Yes \\
$X \setminus (Z(b, d, f) \cup Z(a, e) \cup Z(c, e) \cup Z(c, d))$ & Yes & Yes & Yes \\
$X \setminus (Z(c) \cup Z(b, d) \cup Z(a, e, f))$ & Yes & Yes & Yes \\
$X \setminus (Z(a) \cup Z(b, c, d, f) \cup Z(c, e))$ & Yes & Yes & Yes \\
$X \setminus (Z(b, d) \cup Z(a, c, e, f) \cup Z(c, d) \cup Z(b, f))$ & Yes & Yes & Yes \\
$X \setminus (Z(a, e) \cup Z(a, b, d, f) \cup Z(c, e) \cup Z(c, d))$ & Yes & Yes & Yes \\
$X \setminus (Z(a, e) \cup Z(b, c, d, e) \cup Z(f))$ & Yes & Yes & Yes \\
$X \setminus (Z(b, d, f) \cup Z(c) \cup Z(a, e))$ & Yes & Yes & Yes \\
$X \setminus (Z(e) \cup Z(c, d) \cup Z(a, b, f))$ & Yes & Yes & Yes \\
$X \setminus (Z(b, d) \cup Z(c, e) \cup Z(c, d) \cup Z(a, b, f))$ & Yes & Yes & Yes \\
$X \setminus (Z(c) \cup Z(a, e) \cup Z(a, b, d, f))$ & Yes & Yes & Yes \\
$X \setminus (Z(d) \cup Z(a, c, e) \cup Z(b, f))$ & Yes & Yes & Yes \\
$X \setminus (Z(a) \cup Z(b, c, d, e) \cup Z(f))$ & Yes & Yes & Yes \\
$X \setminus (Z(b, d) \cup Z(a, e, f) \cup Z(c, d) \cup Z(b, f))$ & Yes & Yes & Yes \\
$X \setminus (Z(b, d) \cup Z(f) \cup Z(a, c, d, e))$ & Yes & Yes & Yes \\
$X \setminus (Z(b, d) \cup Z(a, f) \cup Z(a, c, e) \cup Z(b, f))$ & Yes & Yes & Yes \\
$X \setminus (Z(a, f) \cup Z(b, c, d, f) \cup Z(e))$ & Yes & Yes & Yes \\
$X \setminus (Z(a, e, f) \cup Z(c, d) \cup Z(b))$ & Yes & Yes & Yes \\
$X \setminus (Z(b, d) \cup Z(c, e) \cup Z(c, d) \cup Z(a, e, f))$ & Yes & Yes & Yes \\
$X \setminus (Z(b, d) \cup Z(c, e) \cup Z(c, d) \cup Z(a, b, e, f))$ & Yes & Yes & Yes \\
$X \setminus (Z(c, e) \cup Z(c, d) \cup Z(b, d) \cup Z(a, e, f) \cup Z(a, b, f))$ & Yes & Yes & Yes \\
$X \setminus (Z(a) \cup Z(b, c, d, f) \cup Z(e))$ & Yes & Yes & Yes \\
$X \setminus (Z(d) \cup Z(c, e) \cup Z(a, b, f))$ & Yes & Yes & Yes \\
$X \setminus (Z(a, e) \cup Z(c, e) \cup Z(c, d) \cup Z(a, b, f))$ & Yes & Yes & Yes \\
$X \setminus (Z(b, d) \cup Z(f) \cup Z(a, c, e))$ & Yes & Yes & Yes \\
$X \setminus (Z(b, d, f) \cup Z(a, f) \cup Z(a, e) \cup Z(c, e))$ & Yes & Yes & Yes \\
$X \setminus (Z(b, c, d) \cup Z(a, e) \cup Z(f))$ & Yes & Yes & Yes \\
$X \setminus (Z(c, d, e) \cup Z(b, c, d) \cup Z(a, e, f) \cup Z(b, d, f) \cup Z(a, c, e) \cup Z(a, b, f))$ & Yes & Yes & No \\
$X \setminus (Z(c, d, e) \cup Z(b, d, f) \cup Z(b, c, d) \cup Z(a, f) \cup Z(a, e))$ & Yes & Yes & No \\
$X \setminus (Z(c, d, e) \cup Z(b, f) \cup Z(b, d) \cup Z(a, c, e) \cup Z(a, e, f))$ & Yes & Yes & No \\
$X \setminus (Z(c, d) \cup Z(b, d) \cup Z(a, c, e) \cup Z(a, e, f) \cup Z(a, b, f))$ & Yes & Yes & No \\
$X \setminus (Z(b, d, f) \cup Z(a, e, f) \cup Z(c))$ & Yes & Yes & No \\
$X \setminus (Z(b, c, d) \cup Z(a, f) \cup Z(a, c, e) \cup Z(b, f))$ & Yes & Yes & No \\
$X \setminus (Z(b, d, f) \cup Z(a, e, f) \cup Z(c, e) \cup Z(c, d))$ & Yes & Yes & No \\
$X \setminus (Z(b, d, f) \cup Z(a, f) \cup Z(a, e) \cup Z(c, d, e))$ & Yes & Yes & No \\
$X \setminus (Z(b, d, f) \cup Z(a) \cup Z(c, d, e))$ & Yes & Yes & No \\
$X \setminus (Z(c, d, e) \cup Z(b, d) \cup Z(a, e, f) \cup Z(b, f))$ & Yes & Yes & No \\
$X \setminus (Z(c, e) \cup Z(b, c, d) \cup Z(b, d, f) \cup Z(a, e) \cup Z(a, b, f))$ & Yes & Yes & No \\
$X \setminus (Z(b, c, d) \cup Z(e) \cup Z(a, b, f))$ & Yes & Yes & No \\
$X \setminus (Z(b, c, d) \cup Z(f) \cup Z(a, c, e))$ & Yes & Yes & No \\
$X \setminus (Z(d) \cup Z(a, c, e) \cup Z(a, b, f))$ & Yes & Yes & No \\
$X \setminus (Z(b, d) \cup Z(a, c, e) \cup Z(c, d) \cup Z(a, b, f))$ & Yes & Yes & No \\
$X \setminus (Z(c, e) \cup Z(c, d) \cup Z(b, d, f) \cup Z(a, e, f) \cup Z(a, b, f))$ & Yes & Yes & No \\
$X \setminus (Z(a, e, f) \cup Z(c, d, e) \cup Z(b))$ & Yes & Yes & No \\
$X \setminus (Z(b, c, d) \cup Z(a, e) \cup Z(c, e) \cup Z(a, b, f))$ & Yes & Yes & No \\
$X \setminus (Z(c, d, e) \cup Z(b, f) \cup Z(b, c, d) \cup Z(a, f) \cup Z(a, c, e))$ & Yes & Yes & No \\
\hline
$X \setminus (Z(b, d) \cup Z(a, f) \cup Z(c, e))$ & Yes & No & Yes \\
$X \setminus (Z(a, e) \cup Z(c, d) \cup Z(b, f))$ & Yes & No & Yes \\
$X \setminus (Z(b, c, d) \cup Z(a, e) \cup Z(c, d, e) \cup Z(b, f))$ & Yes & No & No \\
$X \setminus (Z(b, d, f) \cup Z(b, c, d) \cup Z(a, f) \cup Z(c, e))$ & Yes & No & No \\
$X \setminus (Z(a, e, f) \cup Z(a, c, e) \cup Z(c, d) \cup Z(b, f))$ & Yes & No & No \\
$X \setminus (Z(b, d) \cup Z(c, e) \cup Z(a, e, f) \cup Z(a, b, f))$ & Yes & No & No \\
$X \setminus (Z(c, d, e) \cup Z(b, c, d) \cup Z(b, d, f) \cup Z(a, e) \cup Z(a, b, f))$ & Yes & No & No \\
$X \setminus (Z(c, d) \cup Z(b, d, f) \cup Z(a, e, f) \cup Z(a, c, e) \cup Z(a, b, f))$ & Yes & No & No \\
$X \setminus (Z(c, d, e) \cup Z(b, f) \cup Z(b, c, d) \cup Z(a, e, f) \cup Z(a, c, e))$ & Yes & No & No \\
$X \setminus (Z(c, d, e) \cup Z(b, d, f) \cup Z(b, c, d) \cup Z(a, f) \cup Z(a, c, e))$ & Yes & No & No \\
$X \setminus (Z(c, d, e) \cup Z(b, d) \cup Z(a, c, e) \cup Z(a, e, f) \cup Z(a, b, f))$ & Yes & No & No \\
$X \setminus (Z(b, d) \cup Z(a, f) \cup Z(a, c, e) \cup Z(c, d, e))$ & Yes & No & No \\
$X \setminus (Z(b, d, f) \cup Z(a, e) \cup Z(c, d) \cup Z(a, b, f))$ & Yes & No & No \\
$X \setminus (Z(c, e) \cup Z(b, c, d) \cup Z(a, e, f) \cup Z(b, d, f) \cup Z(a, b, f))$ & Yes & No & No \\
\hline
$X \setminus (Z(b) \cup Z(f))$ & No & Yes & --- \\
$X \setminus (Z(c) \cup Z(e))$ & No & Yes & --- \\
$X \setminus (Z(b))$ & No & Yes & --- \\
$X \setminus (Z(b) \cup Z(d))$ & No & Yes & --- \\
$X \setminus (Z(a) \cup Z(f))$ & No & Yes & --- \\
$X \setminus (Z(c) \cup Z(d))$ & No & Yes & --- \\
$X \setminus (Z(a) \cup Z(e))$ & No & Yes & --- \\
$X \setminus (Z(d))$ & No & Yes & --- \\
$X \setminus (Z(a))$ & No & Yes & --- \\
$X \setminus (Z(f))$ & No & Yes & --- \\
$X \setminus (Z(e))$ & No & Yes & --- \\
$X \setminus (Z(c))$ & No & Yes & --- \\
$X$ & No & No & --- \\
\hline
\end{longtable}

\subsection{Construction in Proposition \ref{dim4-rays}}

In the following table, the column headings are as in the previous table, with the additional ``\%C'' = estimated percent complete.  This is computed by intersecting the support of the corresponding fan with the unit sphere and computing the fraction of the unit sphere covered, which we have estimated by taking a random sampling of points uniformly distributed over the unit sphere and counting how many points are in the support of the fan.  It should be noted that this is \textbf{not} an invariant of the corresponding toric variety, since any lattice automorphism of the lattice $N$ of one-parameter subgroups induces an automorphism of the corresponding toric variety, but only orthogonal lattice automorphisms will preserve this measure of completeness.

Note that all GIT quotients are automatically complete, so the blanks for the GIT quotients are all 100\%.  On the other hand, as in the previous, we have not computed the fans of the degenerate quotients and so do not know how complete they are.

\begin{tiny}
\begin{longtable}{l | c | c | c | c}
\hline
Subset & ND & GIT & S & \%C \\
\hline
$X \setminus (Z(a, c, e, g) \cup Z(b, f) \cup Z(c, d, g) \cup Z(b, d) \cup Z(a, e, f))$ & Yes & Yes & Yes \\
$X \setminus (Z(a) \cup Z(b, f) \cup Z(c, d, e, g))$ & Yes & Yes & Yes \\
$X \setminus (Z(b) \cup Z(c, d, g) \cup Z(a, e, f))$ & Yes & Yes & Yes \\
$X \setminus (Z(c) \cup Z(a, e) \cup Z(b, d, f, g))$ & Yes & Yes & Yes \\
$X \setminus (Z(b, d, f) \cup Z(c, d) \cup Z(c, e) \cup Z(a, e, g))$ & Yes & Yes & Yes \\
$X \setminus (Z(c, d) \cup Z(b, f, g) \cup Z(b, d) \cup Z(a, c, e) \cup Z(a, e, f, g))$ & Yes & Yes & Yes \\
$X \setminus (Z(b, d, f, g) \cup Z(c, d) \cup Z(c, e) \cup Z(a, e))$ & Yes & Yes & Yes \\
$X \setminus (Z(a) \cup Z(b, d, f, g) \cup Z(c, e))$ & Yes & Yes & Yes \\
$X \setminus (Z(c, d) \cup Z(c, e) \cup Z(b, d, f) \cup Z(a, e, g) \cup Z(a, b, f, g))$ & Yes & Yes & Yes \\
$X \setminus (Z(b, d, f, g) \cup Z(c, e) \cup Z(a, e) \cup Z(b, c, d, g) \cup Z(a, f))$ & Yes & Yes & Yes \\
$X \setminus (Z(c, d) \cup Z(c, e) \cup Z(a, e, f, g) \cup Z(b, d))$ & Yes & Yes & Yes \\
$X \setminus (Z(c, d) \cup Z(b, f) \cup Z(a, e, f, g) \cup Z(b, d))$ & Yes & Yes & Yes \\
$X \setminus (Z(b, d, f, g) \cup Z(a, e) \cup Z(c, e) \cup Z(a, f))$ & Yes & Yes & Yes \\
$X \setminus (Z(a, b, f) \cup Z(c, d, g) \cup Z(c, d, e) \cup Z(b, f, g) \cup Z(a, e))$ & Yes & Yes & Yes \\
$X \setminus (Z(c) \cup Z(a, e, f, g) \cup Z(b, d))$ & Yes & Yes & Yes \\
$X \setminus (Z(b, f) \cup Z(c, d, g) \cup Z(b, c, d) \cup Z(a, e, g) \cup Z(a, e, f))$ & Yes & Yes & Yes \\
$X \setminus (Z(a) \cup Z(c, d, e) \cup Z(b, f, g))$ & Yes & Yes & Yes \\
$X \setminus (Z(c) \cup Z(b, d, f) \cup Z(a, e, g))$ & Yes & Yes & Yes \\
$X \setminus (Z(b, f) \cup Z(b, d) \cup Z(c, d, g) \cup Z(a, e, f))$ & Yes & Yes & Yes \\
$X \setminus (Z(c, d) \cup Z(b) \cup Z(a, e, f, g))$ & Yes & Yes & Yes \\
$X \setminus (Z(a, c, e, g) \cup Z(b, f) \cup Z(b, d) \cup Z(a, f) \cup Z(c, d, e, g))$ & Yes & Yes & Yes \\
$X \setminus (Z(b) \cup Z(a, f) \cup Z(c, d, e, g))$ & Yes & Yes & Yes \\
$X \setminus (Z(c, d) \cup Z(b, f, g) \cup Z(a, c, e) \cup Z(b, d, f) \cup Z(a, e, g))$ & Yes & Yes & Yes \\
$X \setminus (Z(c, d, e) \cup Z(b, f, g) \cup Z(a, e) \cup Z(b, c, d, g) \cup Z(a, f))$ & Yes & Yes & Yes \\
$X \setminus (Z(a, e) \cup Z(b, f) \cup Z(c, d, g))$ & Yes & Yes & Yes \\
$X \setminus (Z(a, e) \cup Z(c, d, e) \cup Z(a, f) \cup Z(b, f, g))$ & Yes & Yes & Yes \\
$X \setminus (Z(b, f) \cup Z(b, c, d) \cup Z(a, f) \cup Z(a, e, g) \cup Z(c, d, e, g))$ & Yes & Yes & Yes \\
$X \setminus (Z(b, f, g) \cup Z(c, d) \cup Z(a, e))$ & Yes & Yes & Yes \\
$X \setminus (Z(b, d, f, g) \cup Z(c, d) \cup Z(c, e) \cup Z(a, e) \cup Z(a, b, f, g))$ & Yes & Yes & Yes \\
$X \setminus (Z(c, d) \cup Z(a, c, e, g) \cup Z(b, f) \cup Z(b, d) \cup Z(a, e, f, g))$ & Yes & Yes & Yes \\
$X \setminus (Z(b, d, f, g) \cup Z(a, b, f) \cup Z(c, e) \cup Z(c, d, g) \cup Z(a, e))$ & Yes & Yes & Yes \\
$X \setminus (Z(b, f) \cup Z(a, e) \cup Z(b, c, d, g) \cup Z(a, f) \cup Z(c, d, e, g))$ & Yes & Yes & Yes \\
$X \setminus (Z(a, e) \cup Z(b, f) \cup Z(a, f) \cup Z(c, d, e, g))$ & Yes & Yes & Yes \\
$X \setminus (Z(c, d) \cup Z(b, f) \cup Z(a, e, g))$ & Yes & Yes & Yes \\
$X \setminus (Z(c, d, e, g) \cup Z(b, f) \cup Z(a, f) \cup Z(b, d))$ & Yes & Yes & Yes \\
$X \setminus (Z(c, d) \cup Z(c, e) \cup Z(b, d) \cup Z(a, e, f, g) \cup Z(a, b, f, g))$ & Yes & Yes & Yes \\
$X \setminus (Z(a, c, e, g) \cup Z(b, f) \cup Z(c, d, g) \cup Z(b, c, d) \cup Z(a, e, f))$ & Yes & Yes & No \\
$X \setminus (Z(c, d) \cup Z(a, c, e) \cup Z(b, d, f) \cup Z(a, e, f, g) \cup Z(a, b, f, g))$ & Yes & Yes & No \\
$X \setminus (Z(b, d, f, g) \cup Z(a, b, f) \cup Z(c, d, e) \cup Z(a, e) \cup Z(b, c, d, g))$ & Yes & Yes & No \\
$X \setminus (Z(a, c, e, g) \cup Z(b, f) \cup Z(b, c, d) \cup Z(c, d, e, g) \cup Z(a, e, f))$ & Yes & Yes & No \\
$X \setminus (Z(a, c, e, g) \cup Z(b, f) \cup Z(c, d, g) \cup Z(b, d) \cup Z(a, e, f, g))$ & Yes & Yes & No \\
$X \setminus (Z(c, d) \cup Z(a, c, e, g) \cup Z(b, f, g) \cup Z(b, d, f) \cup Z(a, e, f, g))$ & Yes & Yes & No \\
$X \setminus (Z(a, b, f) \cup Z(b, f, g) \cup Z(a, e) \cup Z(b, c, d, g) \cup Z(c, d, e, g))$ & Yes & Yes & No \\
$X \setminus (Z(b, f) \cup Z(c, d, g) \cup Z(a, e, g) \cup Z(a, e, f))$ & Yes & Yes & No \\
$X \setminus (Z(b, d, f, g) \cup Z(c, e) \cup Z(c, d, g) \cup Z(a, e) \cup Z(a, b, f, g))$ & Yes & Yes & No \\
$X \setminus (Z(b, f) \cup Z(c, d, g) \cup Z(a, e, g) \cup Z(b, c, d))$ & Yes & Yes & No \\
$X \setminus (Z(a, e) \cup Z(b, c, d, g) \cup Z(b, f) \cup Z(c, d, e, g))$ & Yes & Yes & No \\
$X \setminus (Z(b, f, g) \cup Z(a, e) \cup Z(b, c, d, g) \cup Z(a, f) \cup Z(c, d, e, g))$ & Yes & Yes & No \\
$X \setminus (Z(a, b, f) \cup Z(c, d, e) \cup Z(b, f, g) \cup Z(a, e) \cup Z(b, c, d, g))$ & Yes & Yes & No \\
$X \setminus (Z(c, d) \cup Z(a, c, e) \cup Z(b, d, f) \cup Z(a, e, g) \cup Z(a, b, f, g))$ & Yes & Yes & No \\
$X \setminus (Z(a, e) \cup Z(a, f) \cup Z(b, f, g) \cup Z(c, d, e, g))$ & Yes & Yes & No \\
$X \setminus (Z(c, d) \cup Z(b, f) \cup Z(a, e, f, g) \cup Z(a, c, e, g))$ & Yes & Yes & No \\
$X \setminus (Z(c, d) \cup Z(b, f, g) \cup Z(a, c, e) \cup Z(b, d, f) \cup Z(a, e, f, g))$ & Yes & Yes & No \\
$X \setminus (Z(b, d, f, g) \cup Z(c, d) \cup Z(a, e) \cup Z(a, b, f, g))$ & Yes & Yes & No \\
$X \setminus (Z(c, d, e, g) \cup Z(b, f) \cup Z(a, e, f) \cup Z(b, d))$ & Yes & Yes & No \\
$X \setminus (Z(b, d, f, g) \cup Z(a, e) \cup Z(c, d, e) \cup Z(a, f))$ & Yes & Yes & No \\
$X \setminus (Z(b, d, f, g) \cup Z(c, d) \cup Z(a, c, e) \cup Z(a, e, g) \cup Z(a, b, f, g))$ & Yes & Yes & No \\
$X \setminus (Z(b, d, f, g) \cup Z(c, d, g) \cup Z(c, d, e) \cup Z(a, e) \cup Z(a, b, f, g))$ & Yes & Yes & No \\
$X \setminus (Z(b) \cup Z(c, d, g) \cup Z(a, e, f, g))$ & Yes & Yes & No \\
$X \setminus (Z(a, e, g) \cup Z(c, d) \cup Z(b, f, g) \cup Z(a, c, e))$ & Yes & Yes & No \\
$X \setminus (Z(a, b, f) \cup Z(a, e) \cup Z(c, d, g) \cup Z(b, f, g))$ & Yes & Yes & No \\
$X \setminus (Z(b, f) \cup Z(a, e, f, g) \cup Z(c, d, g) \cup Z(b, d))$ & Yes & Yes & No \\
$X \setminus (Z(a) \cup Z(b, d, f, g) \cup Z(c, d, e))$ & Yes & Yes & No \\
$X \setminus (Z(b, d, f, g) \cup Z(c, d) \cup Z(c, e) \cup Z(a, e, g))$ & Yes & Yes & No \\
$X \setminus (Z(b, d, f, g) \cup Z(a, b, f) \cup Z(c, d, g) \cup Z(c, d, e) \cup Z(a, e))$ & Yes & Yes & No \\
$X \setminus (Z(b, d, f) \cup Z(c, d) \cup Z(c, e) \cup Z(a, e, f, g))$ & Yes & Yes & No \\
$X \setminus (Z(b, d, f, g) \cup Z(c, d, e) \cup Z(a, e) \cup Z(b, c, d, g) \cup Z(a, f))$ & Yes & Yes & No \\
$X \setminus (Z(b) \cup Z(c, d, e, g) \cup Z(a, e, f))$ & Yes & Yes & No \\
$X \setminus (Z(b, f) \cup Z(b, c, d, g) \cup Z(a, f) \cup Z(a, e, g) \cup Z(c, d, e, g))$ & Yes & Yes & No \\
$X \setminus (Z(a, e) \cup Z(c, d, g) \cup Z(c, d, e) \cup Z(b, f, g))$ & Yes & Yes & No \\
$X \setminus (Z(a, c, e, g) \cup Z(b, f) \cup Z(c, d, g) \cup Z(b, c, d) \cup Z(a, e, f, g))$ & Yes & Yes & No \\
$X \setminus (Z(a) \cup Z(b, f, g) \cup Z(c, d, e, g))$ & Yes & Yes & No \\
$X \setminus (Z(b, f, g) \cup Z(c, d, g) \cup Z(a, e, g))$ & Yes & Yes & No \\
$X \setminus (Z(c, d) \cup Z(a, c, e, g) \cup Z(b, f, g) \cup Z(b, d) \cup Z(a, e, f, g))$ & Yes & Yes & No \\
$X \setminus (Z(c) \cup Z(b, d, f) \cup Z(a, e, f, g))$ & Yes & Yes & No \\
$X \setminus (Z(a, c, e, g) \cup Z(b, f) \cup Z(b, c, d) \cup Z(a, f) \cup Z(c, d, e, g))$ & Yes & Yes & No \\
$X \setminus (Z(a, e, g) \cup Z(c, d) \cup Z(b, d, f) \cup Z(b, f, g))$ & Yes & Yes & No \\
$X \setminus (Z(c, d) \cup Z(a, c, e) \cup Z(b, d) \cup Z(a, e, f, g) \cup Z(a, b, f, g))$ & Yes & Yes & No \\
$X \setminus (Z(b, d, f, g) \cup Z(a, b, f) \cup Z(c, e) \cup Z(a, e) \cup Z(b, c, d, g))$ & Yes & Yes & No \\
$X \setminus (Z(b, d, f, g) \cup Z(c, d) \cup Z(c, e) \cup Z(a, e, g) \cup Z(a, b, f, g))$ & Yes & Yes & No \\
$X \setminus (Z(b, f) \cup Z(b, c, d) \cup Z(c, d, e, g) \cup Z(a, e, g) \cup Z(a, e, f))$ & Yes & Yes & No \\
$X \setminus (Z(a, c, e, g) \cup Z(b, f) \cup Z(b, d) \cup Z(c, d, e, g) \cup Z(a, e, f))$ & Yes & Yes & No \\
$X \setminus (Z(c, d) \cup Z(c, e) \cup Z(b, d, f) \cup Z(a, e, f, g) \cup Z(a, b, f, g))$ & Yes & Yes & No \\
$X \setminus (Z(b, f) \cup Z(b, c, d, g) \cup Z(c, d, e, g) \cup Z(a, e, g) \cup Z(a, e, f))$ & Yes & Yes & No \\
$X \setminus (Z(c) \cup Z(b, d, f, g) \cup Z(a, e, g))$ & Yes & Yes & No \\
\hline
$X \setminus (Z(c, d, e) \cup Z(b, f, g) \cup Z(b, d) \cup Z(a, c, e) \cup Z(a, f))$ & Yes & No & Yes & 93.00\% \\
$X \setminus (Z(c, e) \cup Z(c, d, g) \cup Z(b, c, d) \cup Z(b, d, f) \cup Z(a, b, f) \cup Z(a, e, g) \cup Z(a, e, f))$ & Yes & No & Yes & 91.70\% \\
$X \setminus (Z(c, e) \cup Z(c, d, g) \cup Z(b, d) \cup Z(a, b, f) \cup Z(a, e, f))$ & Yes & No & Yes & 94.30\% \\
$X \setminus (Z(c, d, g) \cup Z(c, d, e) \cup Z(a, c, e) \cup Z(b, f, g) \cup Z(b, d) \cup Z(a, b, f) \cup Z(a, e, f))$ & Yes & No & Yes & 94.10\% \\
$X \setminus (Z(c, d, e) \cup Z(b, f, g) \cup Z(b, c, d) \cup Z(b, d, f) \cup Z(a, c, e) \cup Z(a, f) \cup Z(a, e, g))$ & Yes & No & Yes & 93.40\% \\
$X \setminus (Z(a, c, e) \cup Z(b, c, d) \cup Z(a, b, f) \cup Z(a, e, g) \cup Z(a, e, f) \cup Z(c, d, g) \cup Z(c, d, e) \cup Z(b, f, g) \cup Z(b, d, f))$ & Yes & No & Yes & 94.40\% \\
$X \setminus (Z(c, e) \cup Z(b, c, d) \cup Z(b, d, f) \cup Z(a, f) \cup Z(a, e, g))$ & Yes & No & Yes & 92.80\% \\
$X \setminus (Z(c, e) \cup Z(a, f) \cup Z(b, d))$ & Yes & No & Yes & 93.80\% \\
$X \setminus (Z(c, d, e) \cup Z(a, c, e) \cup Z(b, c, d) \cup Z(b, d, f) \cup Z(a, f) \cup Z(a, e, g))$ & Yes & No & No & 93.40\% \\
$X \setminus (Z(c, d, e) \cup Z(a, c, e) \cup Z(b, c, d) \cup Z(b, f, g) \cup Z(a, b, f) \cup Z(b, d, f) \cup Z(a, e, f))$ & Yes & No & No & 92.90\% \\
$X \setminus (Z(c, d, e) \cup Z(a, c, e) \cup Z(b, c, d) \cup Z(b, d, f) \cup Z(a, f))$ & Yes & No & No & 93.10\% \\
$X \setminus (Z(c, d, e) \cup Z(a, c, e) \cup Z(b, c, d) \cup Z(b, f, g) \cup Z(a, b, f) \cup Z(b, d, f) \cup Z(a, e, g) \cup Z(a, e, f))$ & Yes & No & No & 93.90\% \\
$X \setminus (Z(c, e) \cup Z(b, c, d) \cup Z(b, d, f) \cup Z(a, b, f) \cup Z(a, e, f))$ & Yes & No & No & 93.30\% \\
$X \setminus (Z(c, d, e) \cup Z(a, c, e) \cup Z(b, f, g) \cup Z(b, d) \cup Z(a, b, f) \cup Z(a, e, f))$ & Yes & No & No & 93.50\% \\
$X \setminus (Z(b, d, f) \cup Z(c, e) \cup Z(a, f) \cup Z(b, c, d))$ & Yes & No & No & 93.70\% \\
$X \setminus (Z(c, d, g) \cup Z(c, d, e) \cup Z(a, c, e) \cup Z(b, d) \cup Z(a, b, f) \cup Z(a, e, f))$ & Yes & No & No & 93.60\% \\
$X \setminus (Z(a, b, f) \cup Z(c, e) \cup Z(a, e, f) \cup Z(b, d))$ & Yes & No & No & 92.50\% \\
$X \setminus (Z(c, d, e) \cup Z(a, c, e) \cup Z(b, d) \cup Z(a, b, f) \cup Z(a, e, f))$ & Yes & No & No & 92.60\% \\
$X \setminus (Z(c, d, g) \cup Z(c, d, e) \cup Z(a, c, e) \cup Z(b, c, d) \cup Z(b, d, f) \cup Z(a, b, f) \cup Z(a, e, g) \cup Z(a, e, f))$ & Yes & No & No & 94.80\% \\
$X \setminus (Z(c, d, g) \cup Z(c, d, e) \cup Z(a, c, e) \cup Z(b, c, d) \cup Z(b, f, g) \cup Z(a, b, f) \cup Z(b, d, f) \cup Z(a, e, f))$ & Yes & No & No & 93.50\% \\
$X \setminus (Z(c, d, e) \cup Z(a, c, e) \cup Z(b, c, d) \cup Z(b, d, f) \cup Z(a, b, f) \cup Z(a, e, g) \cup Z(a, e, f))$ & Yes & No & No & 92.40\% \\
$X \setminus (Z(c, d, e) \cup Z(b, f, g) \cup Z(b, c, d) \cup Z(b, d, f) \cup Z(a, c, e) \cup Z(a, f))$ & Yes & No & No & 93.10\% \\
$X \setminus (Z(a, f) \cup Z(c, d, e) \cup Z(a, c, e) \cup Z(b, d))$ & Yes & No & No & 93.60\% \\
$X \setminus (Z(c, e) \cup Z(b, c, d) \cup Z(b, d, f) \cup Z(a, b, f) \cup Z(a, e, g) \cup Z(a, e, f))$ & Yes & No & No & 93.30\% \\
$X \setminus (Z(c, e) \cup Z(c, d, g) \cup Z(b, c, d) \cup Z(b, d, f) \cup Z(a, b, f) \cup Z(a, e, f))$ & Yes & No & No & 92.50\% \\
$X \setminus (Z(c, d, g) \cup Z(c, d, e) \cup Z(a, c, e) \cup Z(b, c, d) \cup Z(b, d, f) \cup Z(a, b, f) \cup Z(a, e, f))$ & Yes & No & No & 94.00\% \\
$X \setminus (Z(c, d, e) \cup Z(a, c, e) \cup Z(b, c, d) \cup Z(b, d, f) \cup Z(a, b, f) \cup Z(a, e, f))$ & Yes & No & No & 94.10\% \\
\hline
$X \setminus (Z(b) \cup Z(a, f))$ & No & Yes & --- & --- \\
$X \setminus (Z(a))$ & No & Yes & --- & --- \\
$X \setminus (Z(b, f) \cup Z(a, f))$ & No & Yes & --- & --- \\
$X \setminus (Z(a) \cup Z(b, f))$ & No & Yes & --- & --- \\
$X \setminus (Z(b))$ & No & Yes & --- & --- \\
$X \setminus (Z(c))$ & No & Yes & --- & --- \\
$X \setminus (Z(a, e) \cup Z(c, e))$ & No & Yes & --- & --- \\
$X \setminus (Z(c) \cup Z(a, e))$ & No & Yes & --- & --- \\
$X \setminus (Z(a) \cup Z(c, e))$ & No & Yes & --- & --- \\
$X \setminus (Z(c, d) \cup Z(b))$ & No & Yes & --- & --- \\
$X \setminus (Z(c) \cup Z(b, d))$ & No & Yes & --- & --- \\
$X \setminus (Z(c, d) \cup Z(b, d))$ & No & Yes & --- & --- \\
$X$ & No & No & --- \\
\hline
\end{longtable}
\end{tiny}

\end{document}